\documentclass[a4paper,10pt]{elsarticle}
\usepackage{latexsym}
\usepackage{amssymb}
\usepackage{amsthm}
\usepackage{amsfonts}
\usepackage{amsmath}
\usepackage{indentfirst}
\usepackage{graphicx}
\usepackage{tikz}
\usepackage{hyperref}

\usepackage{epic, eepic}

\newtheorem{theorem}{Theorem}[section]

\newtheorem{corollary}[theorem]{Corollary}
\newtheorem{lemma}[theorem]{Lemma}
\newtheorem{proposition}[theorem]{Proposition}
\newtheorem{conjecture}[theorem]{Conjecture}
\newtheorem{question}[theorem]{Question}
\newtheorem{questions}[theorem]{Questions}

\theoremstyle{definition}
\newtheorem{definition}[theorem]{Definition}
\newtheorem{remark}[theorem]{Remark}
\newtheorem{example}[theorem]{Example}

\numberwithin{equation}{section}
\numberwithin{figure}{section}

\newcommand\clp{\mathcal{P}}
\newcommand\D{\ensuremath{\Delta}}        
\newcommand{\one}{1}
\newcommand{\ist}{\ensuremath{[\sigma,\tau]}}
\newcommand{\ost}{\ensuremath{(\sigma,\tau)}}
\newcommand{\iuw}{\ensuremath{[u,w]}}
\newcommand{\ouw}{\ensuremath{(u,w)}}
\newcommand{\mst}{\ensuremath{\mu(\sigma,\tau)}}
\newcommand{\sig}{\sigma}
\newcommand{\alp}{\alpha}
\newcommand\bet{\beta}
\newcommand\gam{\gamma}
\newcommand\del{\delta}
\newcommand{\op}{\oplus}
\newcommand{\om}{\ominus}
\newcommand{\as}{\alp\op\sig}
\newcommand{\at}{\alp\op\tau}
\newcommand{\rev}[1]{#1^r}
\newcommand{\comp}[1]{#1^c}
\newcommand{\rc}[1]{#1^{rc}}
\newcommand{\emb}{\eta}
\newcommand\ra{\rightarrow}
\newcommand{\covers}{\rightarrow}
\newcommand{\coverslabel}[1]{\stackrel{#1}{\longrightarrow}}
\newcommand{\decrease}[1]{#1^-}
\newcommand\sg{\varsigma}
\newcommand\leqpo{\leq_0}
\newcommand\geqpo{\geq_0}
\newcommand{\twocases}{\mu(\sg_m, \tau_m)^+}
\newcommand{\len}[1]{|#1|}
\newcommand{\rk}[1]{|#1|}
\newcommand{\parts}[1]{\ell(#1)}
\newcommand{\cpp}{\mathcal{CP}}

\begin{document}

\begin{frontmatter}

\title{On the topology of the permutation pattern poset}

\author{Peter R.\ W.\ McNamara\fnref{peter,simons}}
\ead{peter.mcnamara@bucknell.edu}
\ead[url]{http://www.facstaff.bucknell.edu/pm040}

\author{Einar Steingr\'imsson\fnref{einar,irf}}
\ead{einar@alum.mit.edu}
\ead[url]{https://personal.cis.strath.ac.uk/einar.steingrimsson}

\fntext[simons]{This work was partially supported by a grant from the Simons Foundation (\#245597 to Peter McNamara).}

\fntext[irf]{Steingr\'imsson was supported by grant no.\ 090038013 from the Icelandic Research Fund.}

\address[peter]{Department of Mathematics, Bucknell University, Lewisburg, PA 17837, USA}

\address[einar]{Department\ of Computer and Information Sciences, University of Strathclyde, Glasgow G1 1XH, UK}

\begin{abstract}
  The set of all permutations, ordered by pattern containment, forms a poset.  This paper presents the first explicit major results on the topology of intervals in this poset.  We show that almost all (open) intervals in this poset have a disconnected subinterval and are thus not shellable.  Nevertheless, there seem to be large classes of intervals that are shellable and thus have the homotopy type of a wedge of spheres.  We prove this to be the case for all intervals of layered permutations that have no disconnected subintervals of rank 3 or more.  We also characterize in a simple way those intervals of layered permutations that are disconnected.  These results carry over to the poset of generalized subword order when the ordering on the underlying alphabet is a rooted forest.  We conjecture that the same applies to intervals of separable permutations, that is, that such an interval is shellable if and only if it has no disconnected subinterval of rank 3 or more.  We also present a simplified version of the recursive formula for the M\"obius function of decomposable permutations given by Burstein et al.\ \cite{mob-sep}.
\end{abstract}


\end{frontmatter}

\section{Introduction}\label{sec-intro}

An occurrence of a pattern $p$ in a permutation $\pi$ is a subsequence of $\pi$ whose letters appear in the same relative order of size as those in $p$.  For example, the permutation 416325 contains two occurrences of the pattern 231, in 463 and 462.  The origin of the study of permutation patterns can be traced back a long way.  In the 1960s and 70s the number of permutations of length $n$ avoiding (having no occurrence of) any one of the six patterns of length 3 was determined by Knuth \cite[ Exercise~2.2.1.5]{knuth1} and Rogers \cite{rogers-asc-seqs}.  In all of these cases, which are easily seen to fall into two equivalence classes, the numbers in question turn out to be the $n$-th Catalan number.  In a seminal 1985 paper, Simion and Schmidt \cite{simion-schmidt} then did the first systematic study of pattern avoidance, and established, among other things, the number of permutations avoiding any given set of patterns of length 3.  In the last two decades this research area has grown steadily, and explosively in recent years, with several different directions emerging.  Also, many connections to other branches of combinatorics, other mathematics, physics and biology have been developed, in addition to the strong ties to theoretical computer science, in which pattern research also has roots.  One of the early such connections was established in 1990 by Lakshmibai and Sandhya \cite{lakshmibai-sandhya}, who showed that a Schubert variety $X_\pi$ is smooth if and only if $\pi$ avoids both 4231 and 3412.  For a recent comprehensive survey of pattern research see \cite{kitaev-book}, and \cite{steingrimsson-open-problems} for an overview of the latest developments.

It is easy to see that pattern containment defines a poset (partially ordered set) $\clp$ on the set of all permutations of length $n$ for all $n>0$.  This poset is the underlying object of all studies of pattern avoidance and containment.  A classical question about any combinatorially defined poset is what its M\"obius function is, and in \cite{wilf}, Wilf asked what can be said about the M\"obius function of $\clp$.  A generalization of that question concerns the topology of the (order complexes of) intervals in $\clp$, since the M\"obius function of an interval $I=[a,b]$ in $\clp$ equals the reduced Euler characteristic of the topological space determined by the order complex $\D(I)$, whose faces are the chains of the open interval $(a,b)$.  In particular, we would like to know the homology and the homotopy type of intervals in $\clp$.

The first results on the M\"obius function of intervals of $\clp$ were obtained by Sagan and Vatter \cite{sagan-vatter}, who used discrete Morse theory to compute the M\"obius function for the poset of layered permutations; as they pointed out, this poset is easily seen to be isomorphic to a certain poset they studied of compositions of an integer.  Later results about the M\"obius function of $\clp$ have been obtained by Steingr\'imsson and Tenner \cite{ste-tenner} and by Burstein et al.\ \cite{mob-sep}, the latter of which gave an effective formula for the M\"obius function of intervals of separable permutations (those avoiding both of the patterns 2413 and 3142) and reduced the computation for decomposable permutations (those non-trivially expressible as direct sums) to that for indecomposable ones.  Recently, Smith \cite{smith-shelling-fixed-des} (see also \cite{smith-mobius-one-descent}) obtained the first systematic results for several classes of intervals of indecomposable permutations, including those intervals $[1,\pi]$ where $\pi$ is any permutation with exactly one descent.

Although the techniques employed by Sagan and Vatter \cite{sagan-vatter} are frequently used to obtain results about the homotopy type of the intervals studied, they  did not present such results.  Later, in a paper generalizing the results in \cite{sagan-vatter} and those of Bj\"orner \cite{bjorner-subword} and Tomie \cite{tomie-chebyshev},  McNamara and Sagan \cite{mcnamara-sagan} computed the M\"obius function of generalized subword order, using discrete Morse theory, and also determined the homotopy type of all intervals whose underlying poset has rank at most 1.  That, however, does not encompass the case of layered permutations (or any intervals in $\clp$), since the underlying poset for layered permutations consists of the positive integers, under their usual total ordering.

In this paper we present the first explicit major results on the topology of intervals in $\clp$.  However, this only scratches the surface; the poset $\clp$ is clearly very rich in terms of the variety of its intervals and their topology.  Nevertheless, we hope that the results presented here break the ground for further progress.  In fact, this is already happening.  Recently, after the present paper first appeared as a preprint, Jason Smith \cite{smith-shelling-fixed-des} found further results on the topology of intervals in $\clp$, which we describe below.  Although a completely general characterization of the topology of intervals in this poset may be impossible, it seems warranted, given the results so far and the already more substantial results on the M\"obius function, to hope for comprehensive understanding.  That, in turn, is likely to shed light on various other aspects of the study of permutation patterns, as this poset is a fundamental object for all such studies.  

As is conventional in topological combinatorics, we will say that $I$ has a property if the topological space determined by $\D(I)$ has that property.  As is so often the case, our results on the topology of intervals are mostly based on showing that they are shellable.  This implies that these intervals have the homotopy type of a wedge of spheres, where all the spheres are of the top dimension, that is, the same dimension as $\D(I)$, and the homology is thus only in the top dimension.  In that case, the number of spheres equals, up to a sign depending only on rank, the M\"obius function of the interval.

We first characterize those intervals that are disconnected, since an interval with a disconnected subinterval of rank at least 3 is certainly not shellable.  An example of a disconnected interval is given in Figure~\ref{fig-1342-1342675}.
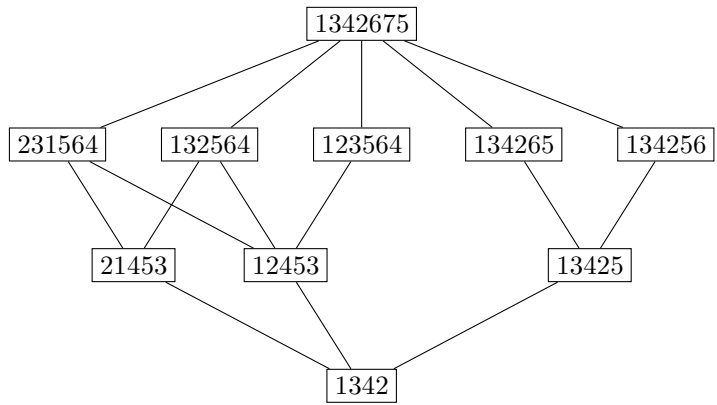
\begin{figure}[htbp]
\begin{center}
\begin{tikzpicture}[scale=2.0]
\tikzstyle{every node}=[rounded corners=0pt, inner sep=3pt]; 
\draw (2,0) node[draw] (a1) {1342};
\draw (0.5,0.8) node[draw] (b1) {21453};
\draw (1.5,0.8) node[draw] (b2) {12453};
\draw (3.5,0.8) node[draw] (b3) {13425};
\draw (0,1.6) node[draw] (c1) {231564};
\draw (1,1.6) node[draw] (c2) {132564};
\draw (2,1.6) node[draw] (c3) {123564};
\draw (3,1.6) node[draw] (c4) {134265};
\draw (4,1.6) node[draw] (c5) {134256};
\draw (2,2.4) node[draw] (d1) {1342675};
\draw (a1) -- (b1);
\draw (a1) -- (b2);
\draw (a1) -- (b3);
\draw (b1) -- (c1);
\draw (b1) -- (c2);
\draw (b2) -- (c1);
\draw (b2) -- (c2);
\draw (b2) -- (c3);
\draw (b3) -- (c4);
\draw (b3) -- (c5);
\draw (c1) -- (d1);
\draw (c2) -- (d1);
\draw (c3) -- (d1);
\draw (c4) -- (d1);
\draw (c5) -- (d1);
\end{tikzpicture}
\end{center}
\caption{The disconnected interval [1342, 1342675]}
\label{fig-1342-1342675}
\end{figure}
If a disconnected subinterval has rank at least 3 we qualify it as being \emph{non-trivial}, since such a subinterval prevents an interval containing it from being shellable, as shown by Bj\"orner \cite[Prop.~4.2]{bjorner}.  (Note that an interval of rank 2 that is not a chain is disconnected, but shellable since its order complex is 0-dimensional.)  It turns out that ``almost all'' intervals in $\clp$ have non-trivial disconnected subintervals and are thus not shellable. More precisely, given any permutation $\sig$, the probability that the interval $\ist$ has such a disconnected subinterval, for a randomly chosen permutation $\tau$ of length $n$, goes to one as $n$ goes to infinity.    Shellable intervals are thus, in this sense, an exception to the general rule.  This seems to be just one manifestation of a more general property of $\clp$: it seems to be very hard to get a grip on its generic intervals.  Even so, there are various substantial classes of intervals where results have been pried out in recent years, and almost certainly more is to come.

We give a very simple characterization of those intervals of layered permutations that 
are disconnected.  This allows us to determine which intervals of layered permutations have no non-trivial disconnected subintervals and, in contrast to statements in the previous paragraph, we show that all such intervals are shellable.  We conjecture that the same is true for intervals of separable permutations, that is, that the only obstruction to shellability of such an interval is a non-trivial disconnected subinterval.

We also present a unified (and simplified) version of the two fundamental propositions in \cite[Propositions~1 and~2]{mob-sep}, which reduce the computation of the M\"obius function for decomposable permutations to a computation involving their components.

As mentioned above, Jason Smith \cite{smith-shelling-fixed-des} has recently found new results on shellability of intervals in $\clp$.   Namely, he shows that intervals of permutations that all have the same number of descents are shellable.  He also conjectures that if $\pi$ has exactly one descent then the interval $[1,\pi]$ is shellable if and only if $\pi$ avoids the two patterns 456123 and 356124.  Containing either of these patterns implies the interval has a disconnected subinterval of rank at least three and thus cannot be shellable.

The paper is organized as follows.  In Section \ref{sec-prelims} we collect some necessary definitions and observations.  In Section \ref{sec-chain} we explain which intervals in $\clp$ are chains.  In Section \ref{non-shell} we show that almost all intervals in $\clp$ are non-shellable, more precisely that for a fixed $\sig$ the proportion of intervals $\ist$ that have non-trivial disconnected subintervals, and are thus non-shellable, goes to one as the length of $\tau$ goes to infinity.  In Section \ref{disc-intervals} we give a general characterization of disconnected intervals in $\clp$. We also show that disconnectivity is preserved under certain operations on intervals and, in Section \ref{sec-iso}, that some of those operations actually give intervals isomorphic to the original ones.  In Section \ref{sec-layered} we give necessary and sufficient conditions for an interval of layered permutations to be disconnected, and show that having no non-trivial disconnected subintervals implies (and hence is equivalent to) shellability.  In fact, our results here apply to a more general situation, namely to generalized subword order (see, for example, \cite{mcnamara-sagan, sagan-vatter}) where the underlying poset is a rooted forest.  In Section \ref{mob-decomp} we give a unified (and simplified) version of the two fundamental recursive formulas in \cite{mob-sep} for the M\"obius function of intervals $\ist$ where $\tau$ is decomposable.  Finally, in Section \ref{sec-open}, we mention some open problems and questions.

\section{Preliminaries}\label{sec-prelims}

In this section, we establish terminology and notation that we will use repeatedly.

The letters of all our permutations $\pi$ are positive integers, and we call the number of letters in $\pi$ the \emph{length} of $\pi$, denoted $\len{\pi}$.  We will use $\emptyset$ to denote the unique permutation of length 0.  As mentioned above, the definition of the partial order in the poset $\clp$ refers only to the relative order of size of letters in permutations. Thus, deleting different letters from a given permutation can result in the same element of $\clp$, such as when we delete either the 2 or the 3 from 416325.  The resulting permutations, 41635 and 41625, are said to be order \emph{isomorphic}, and they have the same \emph{standard form}, namely 31524, since 31524 is the (only) permutation of $\{1,2,3,4,5\}$ whose letters appear in the same order of size as in 41635 and 41625.  The map that takes a permutation to its standard form is referred to as \emph{flattening}.

The \emph{direct sum} of two permutations $\alp$ and $\bet$, denoted $\alp\op\bet$, is the concatenation of $\alp$ and $\bet'$, where $\bet'$ is obtained from $\bet$ by adding to each of its letters the largest letter of $\alp$.  The \emph{skew sum} of $\alp$ and $\bet$, denoted $\alp\ominus\bet$, is the concatenation of $\alp'$ and $\bet$, where $\alp'$ is obtained from $\alp$ by adding to each of its letters the largest letter of $\bet$.  In particular, if $\alp$ and $\bet$ are in standard form, then so are $\alp\op\bet$ and $\alp\ominus\bet$.  For example, if $\alp=213$ and $\bet=3142$, then $\alp\op\bet=2136475$ and $\alp\ominus\bet=6573142$.  We say that a permutation is \emph{decomposable} (respectively \emph{skew decomposable}) if it is the direct sum (resp.\ skew sum) of two nonempty permutations, otherwise it is \emph{indecomposable} (resp.\ \emph{skew indecomposable}).  Clearly, every permutation has a unique \emph{finest} decomposition (resp.\ skew decomposition), that is, a decomposition (resp.\ skew decomposition) into the maximum number of indecomposable (resp.\ skew indecomposable) components.  Note that a permutation cannot be both decomposable and skew decomposable, so every permutation is either indecomposable or skew indecomposable (or both).

For permutations $\sig \leq \tau$ with $\sig$ of length $k$, an \emph{embedding $\emb$ of $\sig$ in $\tau$} is a sequence 
\[
\emb = (0,\ldots,0,\sig(1),0,\ldots,0,\sig(2),0,\ldots,0,\sig(k), 0, \ldots, 0)
\]
of length $\len{\tau}$ so that the nonzero positions in $\emb$ are the positions of an occurrence of $\sig$ in $\tau$.  For example, $21300$, $21030$ and $21003$ are the embeddings of $213$ in $21453$.
A key concept is that every maximal chain from $\tau$ to $\sig$ corresponds to at least one embedding of $\sig$ in $\tau$: starting at $\tau$, each covering relation corresponds to ``zeroing out'' a not necessarily unique letter of $\tau$.  For example, with $\covers$ denoting a covering relation, the chain
\[
21453 \covers 2134 \covers 213
\]
corresponds to the embeddings $21300$ and $21030$ because of the following two choices for zeroing out letters:
\[
21453 \covers 21340 \covers 21300,
\]
\[
21453 \covers 21340 \covers 21030.
\]

To every such embedding $\emb$, we define its \emph{zero set} to be the set of positions that are zero.  Given a permutation $\tau$ and a subset $Z$ of $\{1, \ldots, \len{\tau}\}$, let $\tau-Z$ denote the permutation obtained by deleting the letters of $\tau$ in positions in $Z$ and then flattening.  We will often think of elements of $[\sig,\tau]$ as being of the form $\tau - Z$.  

As always in posets, the closed interval $\ist$ in $\clp$ is the set $\{\pi\;|\;\sig\le\pi\le\tau\}$, and the open interval $(\sig,\tau)$ (the \emph{interior of $\ist$}) is the set $\{\pi\;|\;\sig<\pi<\tau\}$, where ``$<$'' and ``$\le$'' have the usual meaning.  When we talk about topological properties of an interval $I=\ist$ such as connectedness and shellability (to be discussed later), the interval inherits these properties from the topological space determined by the \emph{order complex} of the open interval $(\sig,\tau)$, that is, from the simplicial complex whose faces are the chains of $(\sig,\tau)$.  We denote this order complex by $\D(I)$ or by $\D(\sig,\tau)$.  The \emph{rank} of a closed interval $\ist$ is one less than the maximum possible number of elements in a chain $\sig<\pi_1<\pi_2<\cdots<\pi_k<\tau$ and the rank of an element $\pi\in\ist$ is the rank of $[\sig,\pi]$.  When we talk about the rank of a subinterval $[\sig',\tau']$ or $(\sig',\tau')$, we always mean the rank of the closed interval, although when talking about topological properties of such an interval these are determined by $(\sig,\tau)$, as mentioned above.

We do not state the definition of shellability here since we have already stated the facts we need about shellability: that shellability of an interval completely determines its topology (a wedge of spheres of the top dimension), that disconnected intervals of rank at least 3 are not shellable, and that an interval of a poset is not shellable if it contains a subinterval that is not shellable.  Our main technique for showing shellability will be CL-shellability, for which we give the necessary details in Section~\ref{sec-layered}.  For background on these concepts we refer the reader to \cite{wachs-lectures}.

\section{When is an interval a chain?}\label{sec-chain}

We begin with a classification of those intervals that are chains since it is an obvious question that is not too difficult to answer.  To give the answer, we will need two definitions and a lemma, which will also be useful later (in the proof of Theorem~\ref{thm-augmentation}).

\begin{definition}
A \emph{run} of a permutation $\tau$ is a contiguous subsequence of letters of $\tau$ of the form $(a, a+1, a+2, \ldots, a+k)$ or $(a, a-1, a-2, \ldots, a-k)$.
\end{definition}

For example, 543126 contains disjoint runs of lengths 3, 2 and 1 in that order.
The key to classifying intervals that are chains will be Lemma~\ref{lem-samerun}.  It has been proved earlier by, for example, Homberger \cite{homberger} and Sagan \cite{sagan} but, for the sake of completeness, we give a proof here.

\begin{lemma}\label{lem-samerun}
Positions $i$ and $j$ of a permutation $\tau$ satisfy $\tau - \{i\} = \tau - \{j\}$ if and only if $i$ and $j$ are positions in the same run of $\tau$.  
\end{lemma}

\begin{proof}
  The ``if'' direction is straightforward to check.  For the ``only if'' direction suppose, without loss of generality, that $i<j$.  We consider the case when $\tau(i) < \tau(j)$, with the other case being similar.  Since the $i$-th letters of $\tau - \{i\}$ and $\tau - \{j\}$ are equal, we get that either $\tau(i+1)=\tau(i)$ or $\tau(i+1)-1=\tau(i)$.  The former case is impossible, so we get that positions $i$ and $i+1$ are in the same run.  Now we know that $\tau-\{i+1\} = \tau-\{j\}$, and $\tau(i+1) \leq \tau(j)$ with equality if and only if $i+1=j$.  Therefore, by induction on $j-i$, we conclude that $i$ and $j$ are positions in the same run of $\tau$.
\end{proof}

We now give a definition that allows for a simple characterization of those intervals that are chains.

\begin{definition}
  If $\ist$ is an interval and $\tau=a_1a_2\ldots a_n$ we say that $a_i$ is \emph{removable (with respect to $\sig$)} if removing $a_i$ from $\tau$ yields a permutation in $\ist$.
\end{definition}

Equivalently, $a_i$ is removable if there is an occurrence of $\sig$ in $\tau$ that does not contain~$a_i$.

\begin{proposition}\label{prop-chain}
The interval $\ist$ is a chain if and only if the set of all removable letters of $\tau$ forms a single run.
\end{proposition}

For example, $[21, 51234]$ is a chain of length 3, the removable letters being 1,2,3,4.

\begin{proof}[Proof of Proposition~\ref{prop-chain}]
Suppose $\ist$ is a chain, so removing any particular removable letter from $\tau$ yields the same permutation.  By Lemma~\ref{lem-samerun}, we get that all the removable letters must be in the same run.   The converse follows from Lemma~\ref{lem-samerun} by induction on the length of the run of removable letters.
\end{proof}

Using similar ideas, we can easily determine the complete structure of intervals of rank 2.  The result below follows from Lemma~\ref{lem-samerun}.  Note that if a run contains one removable letter then all its letters are removable.

\begin{proposition}\label{prop-ranktwo}
An interval $\ist$ of rank 2 has exactly $k$ elements of rank 1, where $k$ is the number of runs of $\tau$ consisting of removable letters.  
\end{proposition}

It will be convenient and sensible to restrict some of our later results to intervals of rank at least 3; because of Proposition~\ref{prop-ranktwo} we can do so in the knowledge that rank 2 intervals are well understood.

\section{Almost all intervals are non-shellable}\label{non-shell}

In studying examples of intervals in $\clp$, one quickly realizes that their structure is not simple in general.  One cause for this is stated in the title of this section and is made precise by the results below.  We begin with two preliminary lemmas, the first of which is a straightforward observation.

\begin{lemma}\label{lem-two-sep-occ}
Let $\sig$ be a permutation of length $k\ge2$.  If $\sig$ is indecomposable then $\sig\op\sig$ contains exactly two occurrences of $\sig$.  If $\sig$ is skew indecomposable then $\sig\om\sig$ contains exactly two occurrences of $\sig$.  In either case, the two occurrences consist necessarily of the first and second component, respectively, in the (skew) sum.
\end{lemma}

The next lemma will be needed in both this and later sections.

\begin{lemma}\label{lem-sum-disconn}
Let $\sig$ be a permutation of length $k\ge2$.  If  $\sig$ is indecomposable then the open interval $(\sig,\sig\op \sig)$ is disconnected.  Otherwise, if $\sig$ is skew indecomposable, the open interval $(\sig,\sig\om \sig)$ is disconnected.
\end{lemma}

\begin{proof}
  Suppose $\sig$ is indecomposable and  let $\tau=\sig\op\sig$.  Since $\len{\sig}\ge2$, 
the interior $(\sig,\tau)$ of $[\sig,\tau]$ is nonempty.   We claim that $(\sig,\tau)$ is the disjoint union of the following two sets, and is thus disconnected:
\begin{align*}
  S &= \{\pi\in (\sig,\tau) \;|\; \pi = A\op \rho,~\mbox{and $A$ constitutes the only occurrence of $\sig$ in $\pi$}\},\\
  T &= \{\pi\in (\sig,\tau) \;|\; \pi = \rho\op A,~\mbox{and $A$ constitutes the only occurrence of $\sig$ in $\pi$}\}.
\end{align*}

Note that $\rho$ must be nonempty in both cases, in order for $\pi$ to belong to $(\sig,\tau)$.

Clearly, these sets are disjoint, since one consists of permutations whose only occurrence of $\sig$ is an initial segment, and the other set consists of permutations whose only occurrence of $\sig$ is a final segment, and each of these permutations is strictly longer than $\sig$.

The sets $S$ and $T$ cover $(\sig,\tau)$, because, by Lemma~\ref{lem-two-sep-occ}, $\tau$ has precisely two occurrences of $\sig$,  consisting of the first half of $\tau$ and the second half, respectively.  Thus, removing any subset of letters from the first half of $\tau$ yields a permutation in $T$, removing any letters from the second half of $\tau$ yields a permutation in $S$, while removing letters from both halves of $\tau$ yields a permutation that is not in $\ost$.

If $\sig$ is skew indecomposable, an analogous argument establishes the claim.
\end{proof}

Since every permutation is either indecomposable or skew indecomposable (or both), Lemma~\ref{lem-sum-disconn} can be applied in the proof of the next result.

\begin{theorem}\label{thm-disconnected-subinterval}
Given a permutation  $\sig$, let $P(n)$ be the probability that the interval $[\sig,\tau]$ has a non-trivial disconnected subinterval, where $\tau$ is a randomly chosen permutation of length $n$. Then
\[
\lim_{n\ra\infty}{P(n)} = 1.
\]
Thus almost all intervals $[\sig,\tau]$ in $\clp$ are not shellable.
\end{theorem}  

\begin{proof} We assume $|\sig|\geq3$ since establishing the first assertion for such $\sig$ will clearly also prove it for shorter permutations.  We can also assume that $\sig$ is indecomposable, with the proof for skew indecomposable $\sig$ being similar.  As $n$ grows, let us consider the probability that $\tau$ contains an occurrence of $\sig\op\sig$.

The Marcus-Tardos Theorem\footnote{The Marcus-Tardos Theorem \cite{marcus-tardos}, previously known as the Stanley-Wilf Conjecture, says that the number of permutations of length $n$ that avoid a given pattern $p$ grows exponentially as a function of $n$, whereas the total number of permutations grows much faster, of course.} tells us that this probability will tend to 1, but we can also show this in the following elementary manner.
Let $k=|\sig\op\sig|$.  Clearly, the probability that any given $k$ letters in $\tau$ do not form an occurrence of $\sig\op\sig$ is $1-\frac{1}{k!}$.  Since $\tau$ contains $\lfloor n/k\rfloor$ disjoint subsequences of $k$ letters, we can certainly say that the probability that $\tau$ contains an occurrence of $\sig\op\sig$ is bounded below by
\[
1-\left(1-\frac{1}{k!}\right)^{\lfloor \frac{n}{k}\rfloor}.
\]
Therefore, the interval $\ist$ with $\sig$ indecomposable contains the subinterval $(\sig, \sig\op\sig)$ with probability tending to 1 as $\len{\tau} \to\infty$.  Lemma~\ref{lem-sum-disconn} tells us that these subintervals are disconnected. 

The second assertion then follows from the fact that disconnected intervals of rank at least 3 are not shellable, and from \cite[Prop.~4.2]{bjorner}, which includes the statement that any subinterval of a shellable interval is shellable.
\end{proof}

\section{Disconnectivity of intervals}\label{disc-intervals}

Clear examples of non-shellable intervals $[\sig,\tau]$ are those for which $(\sig,\tau)$ is disconnected with $\len{\tau}-\len{\sig} \geq 3$.  See Figures~\ref{fig-1342-1342675},  \ref{fig-123-356124} and~\ref{fig-123-351624} for such examples.  In fact, for intervals of rank exactly 3, shellability of $(\sig,\tau)$ is equivalent to connectivity since $\D(\sig,\tau)$ is just a graph, that is, one-dimensional.  Moreover, if an interval contains a subinterval that is not shellable, then it is itself not shellable \cite[Prop.~4.2]{bjorner}, as mentioned above, leading to further relevance of disconnected intervals.  For example, the open interval $(123, 1342675)$ is connected, but its open subinterval $(1342, 1342675)$ is not, and thus $(123, 1342675)$ is not shellable.  In fact, ``most'' non-shellable intervals violate shellability because they contain a non-trivial disconnected subinterval, an assertion made precise by Theorem~\ref{thm-disconnected-subinterval}.  Also
 compare Theorem~\ref{thm-layered-shellable}, which shows that, in the case of layered permutations, any non-shellable interval contains a non-trivial disconnected subinterval.  In summary, if we are to study shellability in the permutation pattern poset, the study of disconnectivity is a natural place to start.  

The first examples of disconnected intervals $(\sig,\tau)$ with $\len{\tau}-\len{\sig} \geq 3$ occur when $\len{\tau}=6$.  Two such examples are shown in Figures~\ref{fig-123-356124} and \ref{fig-123-351624} and others follow from Lemma~\ref{lem-sum-disconn}, such as $[321, 321\op321]$.
\begin{figure}[htbp]
\begin{center}
\begin{tikzpicture}[scale=2.0]
\tikzstyle{every node}=[rounded corners=0pt, inner sep=3pt]; 
\draw (1.5,0) node[draw] (a1) {123};
\draw (0,0.8) node[draw] (b1) {4123};
\draw (1,0.8) node[draw] (b2) {3124};
\draw (2,0.8) node[draw] (b3) {1342};
\draw (3,0.8) node[draw] (b4) {2341};
\draw (0,1.6) node[draw] (c1) {45123};
\draw (1,1.6) node[draw] (c2) {35124};
\draw (2,1.6) node[draw] (c3) {24513};
\draw (3,1.6) node[draw] (c4) {34512};
\draw (1.5,2.4) node[draw] (d1) {356124};
\draw (a1) -- (b1);
\draw (a1) -- (b2);
\draw (a1) -- (b3);
\draw (a1) -- (b4);
\draw (b1) -- (c1);
\draw (b1) -- (c2);
\draw (b2) -- (c2);
\draw (b3) -- (c3);
\draw (b4) -- (c3);
\draw (b4) -- (c4);
\draw (c1) -- (d1);
\draw (c2) -- (d1);
\draw (c3) -- (d1);
\draw (c4) -- (d1);
\end{tikzpicture}
\end{center}
\caption{The interval [123, 356124]}
\label{fig-123-356124}
\end{figure}
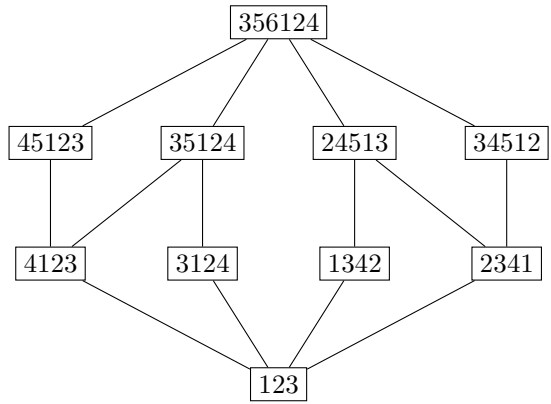
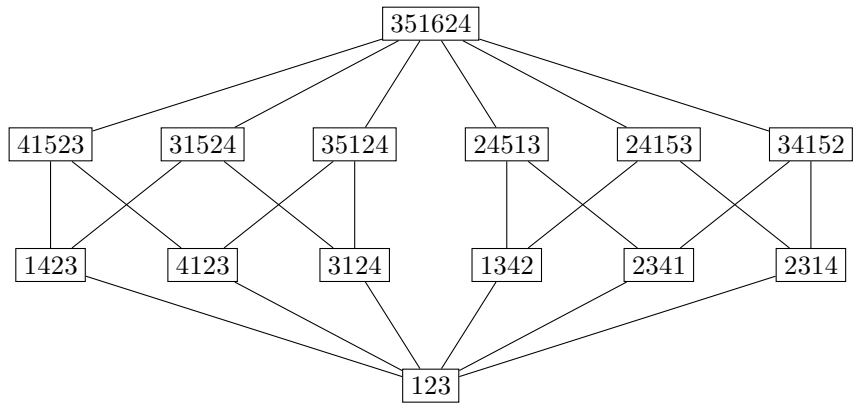
\begin{figure}[htbp]
\begin{center}
\begin{tikzpicture}[scale=2.0]
\tikzstyle{every node}=[rounded corners=0pt, inner sep=3pt]; 
\draw (2.5,0) node[draw] (a1) {123};
\draw (0,0.8) node[draw] (b1) {1423};
\draw (1,0.8) node[draw] (b2) {4123};
\draw (2,0.8) node[draw] (b3) {3124};
\draw (3,0.8) node[draw] (b4) {1342};
\draw (4,0.8) node[draw] (b5) {2341};
\draw (5,0.8) node[draw] (b6) {2314};
\draw (0,1.6) node[draw] (c1) {41523};
\draw (1,1.6) node[draw] (c2) {31524};
\draw (2,1.6) node[draw] (c3) {35124};
\draw (3,1.6) node[draw] (c4) {24513};
\draw (4,1.6) node[draw] (c5) {24153};
\draw (5,1.6) node[draw] (c6) {34152};
\draw (2.5,2.4) node[draw] (d1) {351624};
\draw (a1) -- (b1);
\draw (a1) -- (b2);
\draw (a1) -- (b3);
\draw (a1) -- (b4);
\draw (a1) -- (b5);
\draw (a1) -- (b6);
\draw (b1) -- (c1);
\draw (b1) -- (c2);
\draw (b2) -- (c1);
\draw (b2) -- (c3);
\draw (b3) -- (c2);
\draw (b3) -- (c3);
\draw (b4) -- (c4);
\draw (b4) -- (c5);
\draw (b5) -- (c4);
\draw (b5) -- (c6);
\draw (b6) -- (c5);
\draw (b6) -- (c6);
\draw (c1) -- (d1);
\draw (c2) -- (d1);
\draw (c3) -- (d1);
\draw (c4) -- (d1);
\draw (c5) -- (d1);
\draw (c6) -- (d1);
\end{tikzpicture}
\end{center}
\caption{The interval [123, 351624]}
\label{fig-123-351624}
\end{figure}
\begin{remark}
It is easy to check that Figure~\ref{fig-1342-1342675} gives a negative answer to a question in \cite[\S5]{steingrimsson-open-problems}, which originally appeared (with a typographical error) as \cite[Question~31.2]{mob-sep}.  This question asks if the subcomplex of $\D(\sig,\tau)$ induced by those elements $\pi$ of $\ost$ for which $\mu(\sig,\pi)\neq0$ is a pure complex.  A counterexample with the same rank but with minimal $\len{\tau}$ is $[213,254613]$, which yields a poset isomorphic to the one in Figure~\ref{fig-1342-1342675}.

The motivation for the question above is that the subcomplex might be shellable even if $\D(\sig,\tau)$ is not.  This hope is false even when the subcomplex is pure since $[123,356124]$ is disconnected and remains disconnected after removing the elements $\pi$ for which $\mu(123, \pi)=0$.  See Figure~\ref{fig-123-356124}.

Figure~\ref{fig-123-351624} addresses a side comment from \cite{steingrimsson-open-problems}, which asks for an example of an interval $\ist$ with $\mst = \pm1$, but for which $\D(\sig,\tau)$ is not  homotopy equivalent to a sphere.
\end{remark}

\subsection{A test for disconnectivity of intervals}\label{sub-disconnectivity-test}

Proposition~\ref{prop-disconnectedtest} below gives criteria for checking whether a general interval $\ost$ is disconnected.  Its main application will be its central role in the proof of Theorem~\ref{thm-augmentation}.

Since the lemma below is no more than an expression of the order relation in the permutation pattern poset in terms of zero sets, we omit the proof.

\begin{lemma}\label{lem-zerosets}
Let $\pi_1, \pi_2 \in [\sig,\tau]$ and suppose $\pi_1=\tau-Z_1$ for some $Z_1 \subseteq \{1, \ldots, \len{\tau}\}$.  Then $\pi_1 \leq \pi_2$ if and only if there exists $Z_2 \subseteq Z_1$ such that $\pi_2 = \tau - Z_2$.  
\end{lemma}

For sets $Z_1$ and $Z_2$, we will follow the custom of writing $Z_1 - Z_2$ for the set difference $Z_1 \setminus Z_2$ when $Z_2 \subseteq Z_1$.  

\begin{proposition}\label{prop-disconnectedtest}
Suppose permutations $\sig < \tau$ satisfy $\len{\tau}-\len{\sig} \geq 3$.  Then the open interval $\ost$ is disconnected if and only if the embeddings of $\sig$ in $\tau$ can be partitioned into two nonempty sets $E_1$ and $E_2$ with the following properties:
\begin{enumerate}
\renewcommand{\theenumi}{\alph{enumi}}
\item\label{prop-disc-a} $S_1 \cap S_2 = \emptyset$, where $S_i$ is the union of the zero sets of the elements of $E_i$;
\item\label{prop-disc-b} For all $\emb_1 \in E_1$ with zero set $Z_1$, and all $\emb_2 \in E_2$ with zero set $Z_2$, there do not exist $z_1 \in Z_1$ and $z_2$ in $Z_2$ such that 
\begin{equation}\label{equ-disconnected}
\tau - (Z_1- \{z_1\}) = \tau - (Z_2 - \{z_2\}).
\end{equation}
\end{enumerate}
Furthermore, the nature of the resulting disconnection is that the elements of $(\sig,\tau)$ of the form $\tau - S'_1$ for $S'_1 \subseteq S_1$ are disconnected from those of the form $\tau - S'_2$ for $S'_2 \subseteq S_2$.
\end{proposition}

Before we prove Proposition \ref{prop-disconnectedtest}, we make a few remarks about its content and implications.

Note that Condition (\ref{prop-disc-a}) implies that if $(\sig,\tau)$ is disconnected with $\len{\tau} - \len{\sig} \geq 3$, then $\len{\sig} \geq \len{\tau}/2$. 

Roughly speaking, Condition (\ref{prop-disc-b}) states that we cannot ``add back in'' a single nonzero letter to $\emb_1$ and another to $\emb_2$ to obtain equal permutations.   More precisely, for an interval with a fixed top element $\tau$, we know that permutations in that interval can be identified by their (not necessarily unique) zero sets.  Moreover, embeddings have explicit zero letters.  For an embedding $\emb$ of $\sig$ in $\tau$, we will say that we are \emph{filling a zero} in the embedding when we make a given zero letter of $\emb$ nonzero, thus yielding a unique new permutation $\pi$.  In this terminology,~\eqref{equ-disconnected}  holds if there exist embeddings $\emb_1 \in E_1$ and $\emb_2 \in E_2$ such that filling a zero in each embedding results in the same permutation $\pi$.    
 
\begin{example}
For the interval $[1342, 1342675]$ of Figure~\ref{fig-1342-1342675}, we let 
\[E_1 = \{1342000\} \mbox{\ \ and\ \ } E_2 = \{1000342, 0100342, 0010342, 0001342\}.\]  Clearly Condition (\ref{prop-disc-a}) is satisfied.  To see that (\ref{prop-disc-b}) is satisfied, one can either check that every expression of the form~\eqref{equ-disconnected} is false, or note that $\tau - (Z_1 - \{z_1\})$ will always be $13425$, while any element of the form $\tau - (Z_2- \{z_2\})$ will not have its largest letter at the end.  Thus $(1324, 1342675)$ is confirmed as disconnected.

To see that Condition  (\ref{prop-disc-a}) alone is insufficient to imply disconnectivity, consider the interval $(23514, 24618357)$, where we let $E_1 = \{02305104\}$ and $E_2 = \{23510040\}$ without loss of generality.  This interval satisfies  (\ref{prop-disc-a}) but is connected.  It does not satisfy  (\ref{prop-disc-b}): letting $z_1 = 7$ and $z_2 = 6$, we see that both sides of~\eqref{equ-disconnected} yield 246135.  

To see that Condition  (\ref{prop-disc-b}) alone is insufficient to imply disconnectivity, consider the interval $(12, 45312)$, where we let $E_1 = \{12000\}$ and $E_2 = \{00012\}$ without loss of generality.  We see that this interval satisfies (\ref{prop-disc-b}) but is connected; clearly, it does not satisfy (\ref{prop-disc-a}).  

The condition $\len{\tau}-\len{\sig}\geq 3$ is necessary since, for example, the interval $(\one,213)$ is disconnected but does not satisfy (\ref{prop-disc-a}).  
\end{example}

\begin{proof}[Proof of Proposition~\ref{prop-disconnectedtest}]
Suppose $\ost$ is disconnected and can be partitioned into two subposets $P_1$ and $P_2$ that are not connected to each other.  This induces a partition of the set $\{1,2,\ldots,\len{\tau}\}$ into three sets $S_1$, $S_2$ and $R$, defined in the following way: if $\tau - \{j\} \in P_i$ then $j \in S_i$ for $i \in \{1,2\}$, while otherwise $j \in R$.    For the ``only if'' direction of the proof, we will begin by determining those sets $S$ for which $\tau - S \in P_i$.  

Since the union of $P_1$ and $P_2$ is all of $\ost$, it follows that $\tau - \{r\} \not\geq \sig$ for any $r \in R$.  Thus if $\tau - S \in [\sig,\tau]$ for some subset $S$ of $\{1,2,\ldots,\len{\tau}\}$, then $S$ can only contain elements of $S_1$ and $S_2$.  With this in mind, 
let $a_1 \in S_1$ and $a_2 \in S_2$ and consider $\pi = \tau - \{a_1,a_2\}$, which is covered by both $\tau - \{a_1\}$ and $\tau - \{a_2\}$.  If $\pi \geq \sig$ then, since $\len{\tau}-\len{\sig} \geq 3$, $\pi$ is in both $P_1$ and $P_2$, a contradiction.  Thus $\pi \not\geq \sig$.  Continuing this argument, any element of the form $\tau - (S_1' \cup S_2')$ for nonempty subsets $S_i' $ of $S_i$ will lie below $\tau - \{a_1,a_2\}$ for any $a_1 \in S_1'$ and $a_2 \in S_2'$, and so will not be in $\ist$.  Thus every element of $\ist$ takes the form $\tau - S'_i$ for some $i \in \{1,2\}$ and $S'_i \subseteq S_i$. Moreover, $\pi \in (\sig,\tau)$ takes the form $\pi = \tau - S'_i$ with $S'_i \subseteq S_i$ if and only if $\pi \in P_i$, since $\tau - \{j\} \in P_i$ for all $j \in S'_i$.   

For $i=1,2$, let $E_i$ be the embeddings of $\sig$ in $\tau$ obtained from $\tau$ by deleting only elements of $S_i$.  If an embedding $\emb$ of $\sig$ in $\tau$ were not in either $E_1$ or $E_2$, then $\sig$ could take the form $\tau - (S_1' \cup S_2')$ for nonempty subsets $S_i' $ of $S_i$, contradicting the argument of the previous paragraph.  To see why $S_i$ is the union of the zero sets of the elements of $E_i$, as in the statement of the proposition, we make two observations.  First, the zero set of any element of $E_i$ is contained in $S_i$.  Secondly, with the aim of showing that any element of $S_1$ or $S_2$ is contained in the zero set of some element of $E_1 \cup E_2$, let $j \in S_1 \cup S_2$.  Then $\tau - \{j\} \in P_1 \cup P_2$, and at least one embedding that includes $j$ in its zero set can be obtained by following a maximal chain from $\tau$ to $\sig$ via $\tau - \{j\}$.   Thus $j$ is in the zero set of some element of $E_1 \cup E_2$.  Therefore, each $S_i$ is the union of the zero sets of the elements of $E_i$.  By the definition of $S_i$ at the start of this proof, (\ref{prop-disc-a}) is now immediate.  If (\ref{prop-disc-b}) failed to hold, by the last sentence of the previous paragraph the element given by both sides of~\eqref{equ-disconnected} would be in both $P_1$ and $P_2$, a contradiction.  Thus (\ref{prop-disc-a}) and (\ref{prop-disc-b}) both hold.

Now suppose that (\ref{prop-disc-a}) and (\ref{prop-disc-b}) both hold, and define $S_i$ as in (a).  For $i \in \{1,2\}$, let $P_i$ consist of the elements of $(\sig,\tau)$ of the form $\tau - S_i'$ for some $S_i' \subseteq S_i$.  We wish to show that $P_1$ is disconnected from $P_2$ and that their union is all of $(\sig,\tau)$.  Note that this will automatically give the last assertion of the statement of the proposition.

We first observe that every element $\pi \in (\sig,\tau)$ is in $P_1$ or $P_2$.  Indeed, $\pi$ is on a maximal chain from $\tau$ to $\sig$, and following the edges of this maximal chain will determine at least one embedding of $\sig$ in $\tau$.  Thus $\pi = \tau - S$ for some $S \subseteq S_1$ or some $S \subseteq S_2$, or both.

Towards a contradiction, suppose $(\sig,\tau)$ is connected.  Without loss of generality, there exist $\pi_1 \in P_1$ and $\pi_2 \in P_2$ such that $\pi_1 \geq \pi_2$.  Let $\pi \leq \pi_2$ be a minimal element of $\ost$.  We will show that $\pi$ gives a solution to~\eqref{equ-disconnected}, which will be a contradiction.  Suppose $\pi$ can take the form $\tau - (S_1' \cup S_2')$ for nonempty subsets $S_i' $ of $S_i$.  Therefore, $\pi$ is on a maximal chain from $\tau$ to $\sig$ that zeroes out elements of both $S_1$ and $S_2$, and suppose this maximal chain determines the embedding $\emb$ of $\sig$ in $\tau$.   Since $S_1 \cap S_2 = \emptyset$, $\emb$ cannot be in $E_1$ because $\emb$ has elements of $S_2$ zeroed out, and similarly $\emb \not\in E_2$.  But since every embedding of $\sig$ in $\tau$ is in $E_1 \cup E_2$, we get that $\pi \not\geq \sig$.  Therefore every expression for $\pi$ in the form $\tau - S$ must take the (not necessarily unique) form $\tau - S'_i$ for $S'_i \subseteq S_i$ for some $i \in \{1,2\}$. 

Since $\pi \leq \pi_1$ and $\pi_1$ can take the form $\tau - S'_1$ for $S'_1 \subseteq S_1$, Lemma~\ref{lem-zerosets} gives that $\pi = \tau - S''_1$ for some $S''_1 \subseteq S_1$.  Taking a corresponding maximal chain $C_1$ from $\tau$ to $\sig$ through $\pi_1$ and $\pi$ will thus end at an embedding $\emb_1 \in E_1$ with $\emb_1$ having zero set $Z_1$. Similarly, since $\pi \leq \pi_2$, we get $\pi = \tau - S''_2$ for some $S''_2 \subseteq S_2$, and a maximal chain $C_2$ with a resulting embedding $\emb_2 \in E_2$ with zero set $Z_2$.  Since $\pi$ is a minimal element of $\ost$, it covers $\sig$.  For $i=1,2$, the element of $C_i$ which covers $\sig$ will take the form $\tau - (Z_i - \{z_i\})$ for some $z_i \in Z_i$.   But since this element is $\pi$ for both $C_1$ and $C_2$,~\eqref{equ-disconnected} holds, contradicting (\ref{prop-disc-b}).  We conclude that $\ost$ is disconnected. 
\end{proof}

\begin{corollary}
Suppose permutations $\sig < \tau$ satisfy $\len{\tau}-\len{\sig} \geq 3$.  Then $\ost$ has at least $k$ connected components if and only if the embeddings of $\sig$ in $\tau$ can be partitioned into sets $E_1, E_2, \ldots, E_k$ that pairwise satisfy (\ref{prop-disc-a}) and (\ref{prop-disc-b}) of Proposition~\ref{prop-disconnectedtest}.\end{corollary}

\begin{proof}
For the ``only if'' direction, consider any connected component $P_1$, which we know is disconnected from the remainder $P_2$ of $\ost$, and apply Proposition~\ref{prop-disconnectedtest}.  Since $P_1$ is arbitrary, the result follows.  A similar idea proves the converse.
\end{proof}

An example of an interval $\ost$ with $k$ connected components is given by setting $\sig = 321 \op 321 \op \cdots \op 321$, i.e., the direct sum of $k-1$ copies of 321, and $\tau = \sig \op 321$.  Each connected component is simply a chain of length 1.

\subsection{Preservation of disconnectivity under augmentation}

In practice, many disconnected intervals are of the form $(\as\,,\, \at)$ for some disconnected interval $\ost$.  Our next result explains this phenomenon.
 
\begin{theorem}\label{thm-augmentation}
Suppose $(\sig,\tau)$ is a disconnected interval.  Then for any permutation $\alp$, 
the open interval
\[
(\as\, ,\,\at)
\]
is also disconnected.
\end{theorem}

\begin{proof}
Let us assume that $(\sig,\tau)$ is disconnected.  First suppose that $\len{\tau}-\len{\sig}=2$ and refer to Proposition~\ref{prop-ranktwo}.  If removing a particular letter of $\tau$ gives a permutation greater than $\sig$, then removing the corresponding letter of $\at$ will give a permutation greater than $\as$.  Therefore, there are at least as many elements of rank 1 in $[\as, \at]$ as in $\ist$, implying the result.

Now assume $\len{\tau}-\len{\sig}\geq3$.  We will use Proposition~\ref{prop-disconnectedtest} throughout the remainder of this proof as our characterization of disconnectivity, and adopt the notation used there for the interval $(\sig,\tau)$.  

It will be helpful to use a running example throughout.  Consider the disconnected interval $(321, 326154)$, where $E_1 = \{320100\}$ and $E_2 = \{003021\}$.  Suppose we are trying to show that the interval $(21\op321\,,\, 21\op326154) = (21543, 21548376)$ is disconnected.  Except in sentences where we explicitly mention our example, all statements will apply to the general case.

Our first observation is that when $\as$ embeds in $\at$, the letters of the $\sig$ portion of $\as$ must embed into the $\tau$ portion of $\at$.  In this way, every embedding $\emb^+$ of $\as$ in $\at$ uniquely induces an embedding of $\sig$ in $\tau$.  In our example, the embeddings $21005043$ and $00215043$ both induce the embedding $003021$.  For $i=1,2$, let $E_i^+$ denote those embeddings of $\as$ in $\at$ that induce an element of $E_i$.  In our example, $E_1 = \{21540300\}$ and $E_2 = \{21005043, 00215043\}$.  Clearly, every embedding of $\as$ is in exactly one of $E_1^+$ and $E_2^+$.  

\noindent\textbf{Proof of (\ref{prop-disc-a}).}  Defining $S_i^+$ as the union of the zero sets of the elements of $E_i^+$, our first task is to show that $S_1^+ \cap S_2^+ = \emptyset$.  In the setting of $(\sig,\tau)$, we know that the position 1 cannot be in both $S_1$ and $S_2$ since $S_1 \cap S_2 = \emptyset$.  Suppose, without loss of generality, that $1 \not\in S_1$.  In other words, every embedding $\emb$ in $E_1$ satisfies $\emb(1) \neq 0$.

Let $\emb^+ \in E_1^+$ be an embedding of $\as$ in $\at$ that induces an embedding $\emb \in E_1$.  Since $\emb(1) \neq 0$, it follows from the definition of this inducing that $\emb^+$ must embed the letters of the $\alp$ portion of $\as$ directly into the $\alp$ portion of $\at$.  As a result, the elements of $S^+_1$ are exactly the elements of $S_1$ shifted by an appropriate amount, i.e., 
\begin{equation}\label{equ-s1plus}
S_1^+ = \{s+\len{\alp} : s \in S_1\}.
\end{equation}
In our example, $S_1 = \{3,5,6\}$ and $S_1^+ = \{5,7,8\}$.  In particular, every element of $S_1^+$ is contained in the $\tau$ positions of $\at$, i.e., in the set $\{\len{\alp}+1, \ldots, \len{\alp}+\len{\tau}\}$.  So let us consider the elements of $S_2^+$ contained in the $\tau$ positions of $\at$.  If $\emb^+ \in E_2^+$ induces $\emb \in E_2$, the nonzero letters of $\emb$ in $\tau$ must correspond to nonzero letters of $\emb^+$ in the $\tau$ positions of $\at$.  Therefore, the zero letters of $\emb^+$ in the $\tau$ positions of $\at$ must correspond to zero letters of $\emb$. More precisely, $S_2^+$ is contained in the set
\begin{equation}\label{equ-s2plus}
\{1, 2, \ldots, \len{\alp}\} \cup \{s+\len{\alp} : s \in S_2\}.
\end{equation}
In our example, this is the set $\{1,2\} \cup \{3,4,6\}$.  Comparing~\eqref{equ-s1plus} and~\eqref{equ-s2plus}, the fact that $S_1 \cap S_2 = \emptyset$ implies  that $S_1^+ \cap S_2^+ = \emptyset$.

\

\noindent\textbf{Proof of (\ref{prop-disc-b}).}  The second part of the proof is to show that since $(\sig,\tau)$ satisfies (\ref{prop-disc-b}) of Proposition~\ref{prop-disconnectedtest}, $(\as\,,\, \at)$ satisfies the appropriate analogue.  Suppose to the contrary that 
\begin{equation}\label{equ-plusconnected}
(\at) - (Z_1^+ - \{z_1\}) = (\at) - (Z_2^+ - \{z_2\}),
\end{equation}
where $Z_i^+$ is the zero set of some $\emb_i^+ \in E_i^+$, and $z_i \in Z_i^+$, for $i=1,2$.  In other words, we can fill a zero in each of $\emb_1^+$ and $\emb_2^+$ to obtain a common permutation $\pi$.  Note that $\pi$ covers $\as$.

Again, suppose without loss of generality that $1 \not\in S_1$.  Thus, as before, $\emb_1^+$ embeds the $\alp$ portion of $\as$ directly into the $\alp$ portion of $\at$.  Since $\pi$ is obtained by filling a zero in $\emb_1^+$ and furthermore $1 \not\in S_1$, this zero must be in position $\len{\alp} + j$ of $\emb_1^+$, where
\begin{equation}\label{equ-positionremoved}
j \in \{2, \ldots, \len{\tau}\}.
\end{equation}
In our example, $\emb_1^+ = 21540300$ and suppose we fill the zero in position $2+5$ to obtain $21540360$, so $\pi=215436$.  It follows that 
\begin{equation}\label{equ-pi}
\pi = \alp \op \sig',
\end{equation}
where $\sig'$ is an element that covers $\sig$ in $\ist$.  We know that $\sig'$ can be defined in the setting of $\ost$ in the following way: letting $\emb_1$ denote the embedding of $\sig$ in $\tau$ induced by $\emb_1^+$, $\sig'$ is the permutation obtained from $\emb_1$ by filling the zero in some position $j$, and suppose that the new nonzero entry is the $j'$-th entry of $\sig'$.  Since $1 \not\in S_1$, observe that $1 < j' \leq j$.  In our example, $\emb=320100$, $j=5$, $\sig'=3214$,  and $j'=4$.

Next consider $\emb_2^+$ and suppose first that $\pi$ is obtained from $\emb_2^+$ by filling a zero before the $\len{\alp}$-th nonzero position of $\emb_2^+$.  
Then the first $\len{\alp}+1$ entries of $\pi$ take the form 
\begin{equation}\label{equ-pi2}
(\alp'_1, \ldots, \alp'_{i}, a, \alp'_{i+1}, \ldots, \alp'_{\len{\alp}})
\end{equation}
where $a$ is the entry introduced by the filling of the zero, and each $\alp'_k$ is either equal to $\alp_k$ or $\alp_k+1$, depending on the value of $a$.  In our example, we could take $\emb_2^+ = 00215043$ and if we fill the zero in position 1, we obtain $10326054$; the expression in~\eqref{equ-pi2} is then $132$.  (From this point on, our example becomes less useful since it should not yield a solution to~\eqref{equ-plusconnected}, while our assumption in the general case is that such a solution exists.  Thus the deductions that follow do not apply to our example.) Reconciling~\eqref{equ-pi} and~\eqref{equ-pi2}, we must have $a=\alp_{i+1}$ and $\alp'_{\len{\alp}} > \len{\alp}$, so $\alp'_{\len{\alp}} = \len{\alp}+1$.  Therefore, we get that 
 \[
 \pi = \alp \op \sig' = \alp \op \one \op \sig.
 \]
Thus $\sig' = \one \op \sig$ covers $\sig$ in $[\sig,\tau]$ and has the property that removing its entry in position 1 or position $j'$ recovers $\sig$ after flattening.  Lemma~\ref{lem-samerun} then gives that positions $j'$ and 1 must be in the same run in $\sig'$. 

From $\emb_1^+$, obtain a new embedding $\emb_3^+$ in $\at$ by adding a zero in position $\len{\alp}+1$ and filling the zero in position $\len{\alp}+j$, but otherwise preserving $\emb_1^+$.  Since 1 and $j'$ are positions in the same run in $\sig'$, $\emb_3^+$ is an embedding of $\as$ in $\at$.  Since $\len{\tau}-\len{\sig} \geq 2$, 
$\emb_3^+$ and $\emb_1^+$ have at least one zero position in common, and so $\emb_3^+ \in E_1^+$.  But $\emb_3^+$ induces an embedding $\emb_3$ of $\sig$ in $\tau$ that satisfies $\emb_3(1)=0$.  This contradicts the fact that $1 \not\in S_1$. 

Finally, assume $\pi$ is obtained from $\emb_2^+$ by filling a zero after the $\len{\alp}$-th nonzero position of $\emb_2^+$. 
In our example, we could take $\emb_2^+= 21005043$ and filling the zero in position 6 yields $21006354$ and so $\pi = 216354$.  Alternatively, we could take $\emb_2^+=00215043$ and filling the zero in position 6 yields $00326154$ and so $\pi=326154$.  For $\pi$ to give a solution to~\eqref{equ-plusconnected}, we know from~\eqref{equ-pi} that $\pi$ must take the form $\pi = \alp \op \sig''$ for some $\sig''$.  Letting $\emb_2$ denote the embedding of $\sig$ in $\tau$ induced by $\emb_2^+$, we see that $\sig''$ is an element that covers $\sig$ in $\ist$ and is obtained from $\emb_2$ by filling a zero.  But we already know that $\sig'$ is an element that covers $\sig$ in $\ist$ and is obtained from $\emb_1$ by filling a zero, so~\eqref{equ-pi} implies that $\sig'=\sig''$.
Applying Proposition~\ref{prop-disconnectedtest}(\ref{prop-disc-b}), this contradicts the disconnectivity of $(\sig,\tau)$.  
\end{proof}

For completeness we list in the following corollary the  straightforward symmetric variations of Theorem~\ref{thm-augmentation}.

\begin{corollary}\label{cor-augmentation}
Suppose $(\sig,\tau)$ is a disconnected interval.  Then for any permutation $\alp$, all of the following augmentations of $(\sig,\tau)$ are also disconnected:
\begin{enumerate}
\renewcommand{\theenumi}{\alph{enumi}}
\item $(\as\, ,\,\at)$;
\item $(\alp\om\sig\,,\,\alp\om\tau)$;
\item $(\sig\op\alp\,,\,\tau\op\alp)$;
\item $(\sig\om\alp\,,\,\tau\om\alp)$.
\end{enumerate}
Consequently, any sequence of augmentations from these four types preserves disconnectivity.
\end{corollary}

Combined with Lemma~\ref{lem-sum-disconn} for example, Corollary~\ref{cor-augmentation} allows us to easily generate infinite classes of disconnected intervals.

Our ``augmentation'' terminology is not meant to suggest that the intervals themselves are larger, just that the top and bottom elements of the corresponding closed intervals are longer.  

Recall that the \emph{complement} $\comp{\pi}$ of a permutation $\pi$ with $\len{\pi}=k$ is defined by
\[
\comp{\pi} = (k+1-\pi(1), \ldots, k+1-\pi(k)).
\]

\begin{proof}[Proof of Corollary~\ref{cor-augmentation}]
Part (a) is exactly Theorem~\ref{thm-augmentation}.  
For any permutation $\pi$, let $\rev{\pi}$ denote its reversal, and $\rc{\pi}$ denote the complement of $\rev{\pi}$.  If $(\sig,\tau)$ is disconnected, since reversal and complementation each preserve pattern containment, we know that $(\comp{\sig}, \comp{\tau})$, $(\rc{\sig}, \rc{\tau})$ and $(\rev{\sig}, \rev{\tau})$ are all disconnected.  By Theorem~\ref{thm-augmentation}, we get that the following intervals are all disconnected for any permutation $\alp$:
\[
(\comp{\alp}\op\comp{\sig}, \comp{\alp}\op\comp{\tau}), \ \ 
(\rc{\alp}\op\rc{\sig}, \rc{\alp}\op\rc{\tau}), \ \ 
(\rev{\alp}\op\rev{\sig}, \rev{\alp}\op\rev{\tau}).
\] 
These open intervals can be rewritten as 
\[
(\comp{(\alp\om\sig)}, \comp{(\alp\om\tau)}), \ \ 
(\rc{(\sig\op\alp)}, \rc{(\tau\op\alp)}), \ \ 
(\rev{(\sig\om\alp)}, \rev{(\tau\om\alp)}),
\]
respectively, from which (b), (c) and (d) follow.
\end{proof}

\section{Isomorphism under augmentation}\label{sec-iso}

While Corollary~\ref{cor-augmentation} shows that disconnectivity is preserved under augmentation, under certain conditions we actually get an isomorphism, as we now show.  As in Corollary~\ref{cor-augmentation}, we list here all the  versions obtained from symmetries.

\begin{theorem}\label{thm-isomorphism}
Consider an interval $[\sig,\tau]$ and let $\alp$ and $\gam$ be indecomposable permutations and $\bet$ and $\del$ be skew indecomposable permutations.
\begin{enumerate}
\renewcommand{\theenumi}{\alph{enumi}}
\item If $\alp\op\sig \not\leq \tau$, then $\ist \,\cong\, [\alp\op\sig\,,\, \alp\op\tau]$.
\item If $\bet\om\sig \not\leq \tau$, then $\ist \,\cong\, [\bet\om\sig\,,\,\bet\om\tau]$.
\item If $\sig\op\gam \not\leq \tau$, then $\ist \,\cong\, [\sig\op\gam\,,\, \tau\op\gam]$.
\item If $\sig\om\del \not\leq \tau$, then $\ist \,\cong\, [\sig\om\del\,,\, \tau\om\del]$.
\end{enumerate}
\end{theorem}
In words, each part says that the interval $[\sig,\tau]$ is isomorphic to its augmentation when the augmented interval does not intersect $[\sig,\tau]$.  For example, referring to Figure~\ref{fig-123-351624}, since $\one \op 123 = 1234 \not\leq 351624$, we get $[123,351624] \cong [1234, 1462735]$.  As the proof below shows, the isomorphism simply sends $\pi$ to $\one\op\pi$.

As in Corollary~\ref{cor-augmentation}, a sequence of the augmentations from Theorem~\ref{thm-isomorphism} preserves isomorphism as long as the relevant conditions are satisfied.  For example, for $\alp$ and $\gam$ indecomposable, we get that $[\sig,\tau]\,\cong\, [\alp\op\sig\op\gam\,,\, \alp\op\tau\op\gam]$ as long as $\alp\op\sig \not\leq \tau$ and $\alp\op\sig\op\gam \not\leq \alp\op\tau$, with the latter condition being equivalent to $\sig\op\gam \not\leq \tau$.
For example, 
\[
[321, 321\op 321]\, \cong\, [312\op 321 \op 231\, ,\, 312 \op 321 \op 321 \op 231].
\]
\begin{proof}[Proof of Theorem~\ref{thm-isomorphism}]
We will prove (a).  The other parts are similar, or can be derived from (a) like in the proof of Corollary~\ref{cor-augmentation}. 

Let $\pi \in [\sig,\tau]$.  Then clearly, $\alp\op\pi \in [\alp\op\sig\,,\,\alp\op\tau]$.  
For the other direction, let $\pi \in  [\alp\op\sig\,,\,\alp\op\tau]$.  We wish to show that $\pi$ is of the form $\alp\op \rho$ for some permutation $\rho \in [\sig,\tau]$.  

Since $\alp\op\sig \leq \pi$ but $\alp\op\sig \not\leq \tau$, when $\pi$ embeds into $\alp\op\tau$, some letters of $\pi$ must embed into the $\alp$ portion of $\alp\op\tau$.  So suppose $\pi = \alp' \op \rho$ where $\emptyset < \alp' \leq \alp$ and $\rho \leq \tau$.  Since $\alp\op\sig \leq \pi$, we have $\alp\op\sig = \alp_1 \op \alp_2 \op \sig$ with $\emptyset \leq \alp_1 \leq \alp'$ and $\alp_2 \op \sig \leq \rho$.  Because $\alp$ is indecomposable, we require $\alp_1 = \alp$ or $\alp_2 = \alp$.   In the latter case, we get $\alp\op\sig \leq \rho \leq \tau$, a contradiction.  Thus $\alp' \leq \alp = \alp_1 \leq \alp'$, and so $\alp'=\alp$ and $\pi = \alp \op \rho$, as required.

We conclude that there is a bijection from $[\sig,\tau]$ to $[\alp\op\sig\,,\,\alp\op\tau]$ that sends $\pi$ to $\alp\op\pi$.  It is easy to check that this bijection is order-preserving.  
\end{proof}

Theorem~\ref{thm-isomorphism}(a) does not identify all isomorphisms of the form $\ist\cong [\alp\op\sig\,,\, \alp\op\tau]$.  As a basic example, we have $[\one, 12] \cong [12, 123]$.  The same is true even if we restrict to disconnected intervals, with $[1324, 1365724] \cong [\one \op 1324\,,\, 1 \op 1365724]$ serving as an example.  Looking at this latter isomorphism, one might wonder if it is often the case that 
\[
[\one^k \op \sig\,,\, \one^k \op \tau]\,  \cong\, [\one^{k+1} \op \sig \,,\, \one^{k+1} \op \tau]
\]
for sufficiently large $k$, where $\one^k$ denotes $1\op1\op\cdots\op1$ with $k$ copies of $\one$.  The next result shows that intervals $[\one^k \op \sig\,,\, \one^k \op \tau]$ eventually stabilize as $k$ increases.

\begin{proposition}\label{prop-addingones}
For any interval $\ist$, we have
\begin{equation}\label{equ-addingones}
[\one^k \op \sig\,,\, \one^k \op \tau]\,  \cong\, [\one^{k+1} \op \sig \,,\, \one^{k+1} \op \tau]
\end{equation}
whenever $k \geq \len{\tau}-\len{\sig}-1$.  In fact, if $\tau$ takes the form $\one^\ell \op \tau'$ for some $\tau'$, then \eqref{equ-addingones} holds whenever $k \geq \len{\tau}-\len{\sig}-\ell-1$.
\end{proposition}

\begin{proof}
We will prove the latter assertion since it implies the former.  
First observe that permutations $\pi$ and $\pi'$ satisfy $\pi \leq \pi'$ if and only if $\one \op \pi \leq \one \op \pi'$.  Therefore the map that sends $\pi$ to $\one\op\pi$ will give the desired isomorphism whenever it is surjective.  

So suppose we have an element $\pi$ of $[\one^{k+1} \op \sig \,,\, \one^{k+1} \op \tau]$ that is not of the form $\pi = \one\op \pi'$ for some $\pi'$, i.e., $\pi(1) \neq 1$.  Thus when $\pi$ embeds in $\one^{k+1}\op\tau = \one^{k+1}\op\one^\ell\op\tau'$,  it must embed entirely in $\tau'$, implying that $\len{\pi} \leq \len{\tau}-\ell$.  Also, since $\pi > \one^{k+1} \op \sig$, we know that $\len{\pi} > k+1+\len{\sig}$.  Consequently, we have $k+1+\len{\sig} < \len{\tau}-\ell$, and the result follows.
\end{proof}

The bound on $k$ in Proposition~\ref{prop-addingones} is sharp in the sense that there exist cases where $[\one^k \op \sig\,,\, \one^k \op \tau]$ and $[\one^{k+1} \op \sig \,,\, \one^{k+1} \op \tau]$ are not isomorphic when $k=\len{\tau}-\len{\sig}-\ell-2$.  One example with $\ell=0$ is given by $\ist = [132, 213465]$ since
\[
[\one\op132\,,\, \one\op 213465]\, \not\cong\, [\one \op \one\op 132\,,\, \one \op\one\op 213465]
\]
essentially caused by the fact that $213465$ is an element of the latter interval.

\section{Layered permutations and generalized subword order}\label{sec-layered}

The goal of this section is to completely determine disconnectivity and shellability conditions for intervals of layered permutations.  In contrast with Theorem~\ref{thm-disconnected-subinterval}, we will give an infinite
 class of intervals that are shellable.  In fact, our technique will carry through to the more general case of intervals $[u,w]$ in generalized subword order when the ordering on the alphabet $P$ consists of a rooted forest.  We begin with the necessary preliminaries.

\begin{definition}
A permutation is said to be \emph{layered} if the letters of each component of its finest decomposition are decreasing.  
\end{definition}

For example, $32165798 = 321 \op 21 \op \one \op 21$ is layered.  
We see that every layered permutation is uniquely determined by its composition of layer lengths; it will be helpful to think of layered permutations in terms of these compositions. 

To put these compositions in a more general setting, let $P$ be a poset and let $P^*$ denote the set of finite words in the alphabet consisting of the elements of $P$.  We define generalized subword order on $P^*$ as follows.

\begin{definition}
Let $P$ be a poset.  For $u, w \in P^*$, we write $u \leq w$ and say that $u$ is less than or equal to $w$ in \emph{generalized subword order} if there exists a subword $w(i_1)w(i_2)\cdots w(i_k)$ of the same length as $u$ such that
\[
u(j) \leq_P w(i_j) \mbox{\ \ for all $j$ with $1\leq j \leq k$}. 
\]
\end{definition}
Note that we compare $u(j)$ and $w(i_j)$ in the inequality above according to the partial order $P$.  For example, if $P$ is an antichain, then generalized subword order on $P^*$ is equivalent to ordinary subword order.  More importantly for us, if $P$ is the usual order $\mathbb{P}$ on the positive integers, then generalized subword order amounts to pattern containment order on layered permutations.  For example, with $P = \mathbb{P}$, that $112 \leq 3212$ in generalized subword order is equivalent to the inequality $\one \op \one \op 21 \leq 321 \op 21 \op \one \op 21$ for layered permutations, i.e., $1243 \leq 32165798$.  

We will work in the language of generalized subword order throughout the remainder of this section, referring to layered permutations, or equivalently to the $P=\mathbb{P}$ case, from time to time.  Let us introduce some new notation and translate some of our previous notation and terminology to this generalized subword setting.  We will use $P$ throughout to denote our ordered alphabet, and let $P_0$ denote $P$ with a bottom element $0$ adjoined. We will use $\leqpo$ to denote an inequality in $P_0$, and the symbol $\leq$ without a subscript, when applied to words, will represent an inequality in $P^*$.  We will typically use $u$ and $w$ in place of $\sig$ and $\tau$, $\parts{w}$ will denote the number of letters 
of $w$, and $\rk{w}$ will denote the \emph{rank} of $w$ in $P^*$, which is equal to the sum of the ranks of the letters of $w$ in $P_0$.  For example, with $P=\mathbb{P}$, $\parts{3212}=4$ and $\rk{3212} = 8$, which is consistent with the notation $\len{32154687}=8$ for the corresponding layered permutation.  Ranks are defined in the usual way in $P_0$ since we will hereafter restrict to the case where $P$ is a \emph{rooted forest}, meaning that it consists of a disjoint union of trees, each rooted at a unique bottom element.  Equivalently, every element of $P_0$ except $0$ covers exactly one element.  Note that $P$ being a rooted forest includes the cases when $P$ is an antichain or a chain.  

The notion of embedding for compositions will not be an exact extension of the version for layered permutations.  Instead, suppose $u$ and $w$ are words in $(P_0)^*$.  Then $\emb$ is an \emph{embedding} of $u$ in $w$ if $\emb$ is a word in $(P_0)^*$ obtained from $u$ by inserting $\parts{w}-\parts{u}$ zeros such that $\emb(i) \leqpo w(i)$ for $1 \leq i \leq \parts{w}$.  For example, with $P = \mathbb{P}$, $112$ has three embeddings in $32120$, namely $01120$, $10120$ and $11020$.  If there is more than one embedding of $u$ in $w$, then there is always one embedding $\rho$ that is \emph{rightmost}, defined as follows: if $\emb$ is another embedding, and $\rho(i)$ and $\emb(j)$ both correspond to the same letter of $u$, then $i \geq j$.  For example, with $P = \mathbb{P}$, the rightmost embedding of $112$ in $32120$ is $01120$.  

Our first of two main results of this section gives conditions for an open interval $\ouw$ in $P^*$ to be disconnected.  The only implication we will need for later proofs is that (2)$\Rightarrow$(1), which can be proved as (2)$\Rightarrow$(3)$\Rightarrow$(1) without requiring any further preliminaries; the full details are in the relevant portions of the proof below.  However, we need that (1) implies (2) for the assertion we make immediately before Question~\ref{que-decomposable} and, more to the point, a characterization of disconnectivity in the current case is important for its own sake.  A feature of our proof that (1)$\Rightarrow$(3) is that it requires results from \cite{mcnamara-sagan, sagan-vatter} that rely on Forman's discrete Morse theory.  For the relevant background on discrete Morse theory in the current setting, we refer the reader to \cite[\S2]{mcnamara-sagan} for the bare bones or to \cite[\S4]{sagan-vatter} for more of the topological context.  Readers interested in more general background should consult Forman's papers \cite{forman-dmt, forman-dmt2, forman-users-guide}, and Babson and Hersh \cite{babson-hersh} for the theory applied to order complexes of posets.  Next, we describe the ordering of the maximal chains used in \cite{mcnamara-sagan, sagan-vatter}.

We will order the chains lexicographically according to their edge labels, where we always read along chains from top to bottom.  So let us describe how to label the edges of a maximal chain $C$ in an interval $\iuw$ of $P^*$.  Since the edge labels along $C$ will depend on an embedding of each element of $C$ in $w$, we will first identify a canonical such embedding to ensure that the labeling is well defined.  For elements $v$ and $v'$ of $C$ with $v'$ covering $v$, denoted $v' \covers v$, if $v$ and $v'$ have the same number of letters then there is a unique embedding of $v$ in $v'$.  If $\parts{v} = \parts{v'} -1$, then $v$ is obtained from $v'$ by deleting a letter $a$ that is minimal in $P$.  If this $a$ appears in a consecutive sequence of $a$'s that is maximal under containment, then deletion of any of these $a$'s will also yield $v$.  Our convention in this situation will be to always delete the leftmost $a$ in the sequence.  One can check that, equivalently, the resulting embedding of $v$ in $v'$ is the rightmost embedding, although we will not need that fact.  Working from $w$ down $C$, this process defines a canonical embedding of $v$ in $v'$ for each covering relation, and thus inductively defines a canonical embedding of $v$ in $w$ for any element $v$ of $C$.  These latter embeddings depend on $C$, and it will often be convenient to think of $C$ in terms of the embeddings of its elements in $w$, rather than in terms of the elements themselves.  See~\eqref{equ-position-labels} below for an example, where the labels on the edges will be explained next.

A natural chain labeling of $\iuw$ would label the edge $v' \covers v$ along $C$ by the position in $w$ that is decreased or deleted according to the convention of the previous paragraph.  For example, with $P=\mathbb{P}$, 
\begin{equation}\label{equ-position-labels}
3212 \coverslabel{1} 2212 \coverslabel{2} 2112 \coverslabel{2} 2012 \coverslabel{1} 1012.
\end{equation}
This is exactly the labeling used in \cite{sagan-vatter}, and we will call it the \emph{position labeling}.
In \cite{mcnamara-sagan}, the edge labels are pairs $(i,j)$ where $i$ denotes the position to be decreased and $j$ refers to the new letter in that position; since our $P$ is a rooted forest, it turns out that the second label $j$ is unnecessary and the labeling is equivalent to the position labeling.

In the last part of the following theorem, \emph{a minimal skipped interval} is a notion from discrete Morse theory, which is explained for the current context in \cite[\S4]{sagan-vatter}.

\begin{theorem}\label{thm-layered-disconnected} Let $P$ be a rooted forest.  For $u, w \in P^*$ with $\rk{w}-\rk{u} \geq 3$, the following are equivalent:
\begin{enumerate}
\item $\ouw$ is disconnected;
\item\label{item-layered-disconnected} $u$ and $w$ are the concatenations $u = (v_1, a, v_2)$ and $w= (v_1, a, a, v_2)$ for some letter $a \in P$ and for $v_1, v_2 \in P^*$;
\item\label{item-layered-disconnected-embedding} there exists an embedding $\emb$ of $u$ in $w$ such that, for some $i$, $\emb(i)=0$, $w(i-1)=w(i)$ and $w(j)=\emb(j)$ for $j \neq i$; 
\item under the position labeling, $\ouw$ contains a minimal skipped interval (MSI) with the maximal possible number of elements, i.e., $\rk{w}-\rk{u}-1$ elements;
\end{enumerate}
\end{theorem}

Note that item (\ref{item-layered-disconnected}) in Theorem \ref{thm-layered-disconnected} implies that for an interval $\ist$ of layered permutations with $|\tau|-|\sig|\geq3$ to be disconnected, the composition of $\sig$ is obtained from the composition of $\tau$ by deleting a component that has size at least 3 and that is equal to its preceding component in $\tau$.
An example of this is $[215436,215438769]$, with corresponding compositions 231 and 2331,
 respectively.

\begin{proof}[Proof of Theorem~\ref{thm-layered-disconnected}]
We will show that (1)$\Rightarrow$(4)$\Rightarrow$(3)$\Rightarrow$(1), but let us first show that (2)$\Leftrightarrow$(3).  If $u = (v_1, a, v_2)$ and $w= (v_1, a, a, v_2)$, then one embedding of $u$ in $w$ takes the form $(v_1, a, 0, v_2)$, implying (3).  Conversely, (3) implies that $u$ can be obtained from $w$ by deleting a letter that equals its immediate predecessor, which is equivalent to (2).  

We next show that (1)$\Rightarrow$(4).  Suppose $(u,w)$ is disconnected and is the disjoint union of subposets $Q_1$ and $Q_2$.  Then for any poset lexicographic order of the maximal
chains of $\iuw$, suppose without loss of generality that the lexicographically first chain (reading edge labels from top to bottom) has its interior elements in $Q_1$.  If $C$ is the lexicographically first chain with its interior elements in $Q_2$, then the set of all interior elements of $C$, denoted $C(w,u)$, forms a single MSI by definition of MSI.  Clearly there are $\rk{w}-\rk{u}-1$ elements in $C(w,u)$.

To show (4)$\Rightarrow$(3), suppose $C(w,u)$ is an MSI for some maximal chain $C$ of $\iuw$.  Since $\rk{w}-\rk{u}\geq3$, $C(w,u)$ has at least two elements.  By \cite[Lemma~5.3]{sagan-vatter}, the labels along $C$ cannot contain a descent, since otherwise $C$ would have a single-element MSI, contradicting the fact that $C(w,u)$ is an MSI.  By \cite[Prop.~5.7]{sagan-vatter}, the labels along $C$ cannot contain an ascent, since otherwise $C$ would not be critical, contradicting the fact that $C(w,u)$ is an MSI that contains all the interior elements of $C$.  Therefore, along $C$, just a single position $i$ of $w$ is decreased in going to $u$, and let $\emb$ be the resulting embedding of $u$ in $w$.  This puts us in the setting of \cite[Prop.~3.8]{mcnamara-sagan}, which classifies the MSIs of $P^*$ when a single position is decreased. That proposition gives us the following two relevant facts when $P$ is a rooted forest.  The first is that $\emb$ is not the rightmost embedding.  Then \cite[Lemma~3.7]{mcnamara-sagan} tells us that $\emb(i)=0$ and $w(i-1) \leqpo w(i)$.  The second fact is that $w(i-1)$ cannot be strictly below $w(i)$ in $P_0$.  We conclude that $w(i-1) = w(i)$, and we have arrived at (3).

Finally, we show (3)$\Rightarrow$(1).  Let $C$ be the maximal chain that obtains $u$ from $w$ 
by reducing position $i$ to 0, i.e., by reducing $w(i)$ repeatedly until it becomes a minimal element of $P$, and then deleting that minimal element.  We say that $C$ \emph{zeroes out} position $i$. Since $\rk{w}-\rk{u}\geq3$, we know $w(i)=w(i-1)$ is not a minimal element of $P$, and so $C$ obeys the convention of always zeroing out the leftmost position in a consecutive sequence of some minimal element of $P$.    Under the position labeling, let $Q_1$ consist of all those elements on maximal chains of $[u,w]$ whose first label (at the top) is less than $i$.  Note that $Q_1$ is nonempty since the chain $C'$ that zeroes out position $i-1$ of $w$ is a maximal chain from $w$ down to $u$ that is contained in $Q_1$.  (If $C'$ does not obey our convention about zeroing out positions, then there will be another chain in $Q_1$ that does obey the convention.)  Similarly, let $Q_2$ consist of all those elements on maximal chains of $[w,u]$ whose first label is at least $i$.  In particular, $C$ is contained in $Q_2$.  We wish to show that $Q_1$ and $Q_2$ intersect only at $w$ and $u$.  

Let $C_1 \in Q_1$ and $C_2 \in Q_2$ be arbitrary.  We know that 
\[
u = w(1)\cdots w(i-1) w(i+1) \cdots w(\parts{w}).
\]  
We also know $C_1$ starts at the top by reducing $w(j)$ for some $j<i$. To eventually arrive at $u$, $C_1$ must zero out position $j$.  Since $\parts{u} = \parts{w}-1$, $C_1$ can zero out only one position.  Thus $C_1$ cannot reduce any of the portion $w(i+1)\cdots w(\parts{w})$ if it is to eventually arrive at $u$.  Similarly, $C_2$ zeros out some $w(j)$ for $j \geq i$ and cannot reduce any of the portion $w(1)\cdots w(i-1)$. 

Let $v$ be the first element strictly below $w$ at which $C_1$ and $C_2$ intersect.  We wish to show that $v=u$.  Since $v \in \iuw$, we know $v$ has $\parts{w}$ or $\parts{w}-1$ letters.  By the discussion of the previous paragraph and since $v \in C_1 \cap C_2$, either $v=u$ or $v$ takes the form
\[
v = w(1)\cdots w(i-1)\, a\, w(i+1) \cdots w(\parts{w}),
\]
with $a \in P$.  In the latter case, $v$ can only be obtained from $w$ by reducing $w(i)$, contradicting the fact that $v \in C_1$.   We conclude that $v=u$ as required.    
\end{proof}

As in Theorem~\ref{thm-disconnected-subinterval}, we know that if an interval
 contains a non-trivial disconnected subinterval, then it is not shellable.   It is natural to ask which intervals $[u,w]$ in $P^*$ without such disconnected subintervals are shellable.  Our second main result of this section, Theorem~\ref{thm-layered-shellable}, tells us that when $P$ is a rooted forest, all such intervals are shellable.  This result is a companion to a result from \cite{mcnamara-sagan}, which states that if $P_0$ is finite and has rank at most 2, then any interval in $P^*$ is shellable.  

In order to use Theorem~\ref{thm-layered-shellable} as a test for shellability, we will first extract from Theorem~\ref{thm-layered-disconnected} a criterion for an interval $\iuw$ to contain no disconnected \emph{subintervals}.  It is simpler to state the negated version of such a result.

\begin{proposition}\label{pro-disconnected-subinterval}
Let $P$ be a rooted forest.  An interval $\iuw$ in $P^*$ contains a non-trivial disconnected subinterval if and only if there exits an embedding $\emb$ of $u$ in $w$, an element $a \in P_0$ of rank at least 3, and positions $i<j$, that satisfy all of the following conditions:
\begin{itemize}
\item $w(i) \geqpo a \geqpo \emb(i)$;
\item $w(j) \geqpo a$;
\item $\emb(i+1),\ldots, \emb(j)$ are all zero.
\end{itemize}
\end{proposition}

For example, with $P=\mathbb{P}$, the interval $[121, 23141]$ satisfies the conditions of the proposition by taking $\emb=12001$, $a=3$, $i=2$ and $j=4$.  The disconnected subinterval given in the first paragraph of the proof below is $[131, 1331]$ (there is one other, namely $[231,2331]$).  On the other hand, $[141, 23141]$ has no disconnected subintervals since we can readily check that none of the three embeddings $10041$, $01041$, $00141$ satisfies the conditions of the proposition.  Thus Theorem~\ref{thm-layered-shellable} will tell us that $[141, 23141]$ is shellable.  

\begin{proof}[Proof of Proposition~\ref{pro-disconnected-subinterval}]
If the conditions are satisfied, let $\emb'$ be the embedding such that $\emb'(i)=a$ and $\emb'(k)=\emb(k)$ for $k \neq i$.  Let $\emb''$ be defined by $\emb''(i)=\emb''(j)=a$ and $\emb''(k)=\emb(k)$ for $k \neq i,j$.  Let $u'$ and $w'$ be the permutations corresponding to $\emb'$ and $\emb''$ respectively.  Then $u \leq u' < w' \leq w$, and $u'$ and $w'$ satisfy~(\ref{item-layered-disconnected}) of Theorem~\ref{thm-layered-disconnected}.  Therefore $[u',w']$ is a non-trivial disconnected subinterval of $[u,w]$.  

Now suppose $[u',w']$ is a non-trivial disconnected subinterval of $[u,w]$.  In particular, $|w'|-|u'| \geq 3$.  By (\ref{item-layered-disconnected-embedding}) of Theorem~\ref{thm-layered-disconnected}, there exists an embedding $\rho$ of $u'$ in $w'$ and a position $r$ such that $\rho(r)=0$, $w'(r-1)=w'(r)=\rho(r-1)=a$ and $\rho(k)=w'(k)$ for $k\neq r$.  Since $|w'|-|u'| \geq 3$, we get that $a$ has rank at least 3 in $P_0$.  When we embed $w'$ in $w$, suppose $w'(r-1)$ matches up with $w(i)$ and $w'(r)$ matches up with $w(j)$ for some $i,j$.  Such an embedding, along with $\rho$, induces an embedding $\emb'$ of $u'$ in $w$.  We see that we must have $\emb'(i) = a$ and $\emb'(i+1), \ldots \emb'(j)$ all equal to zero.  Then $\emb'$ along with any embedding of $u$ in $u'$ will induce an embedding $\emb$ of $u$ in $w$ that satisfies the conditions of the proposition.
\end{proof}

We will prove shellability using the notion of CL-shellability, introduced by Bj\"orner and Wachs \cite{bjorner-wachs-coxeter-groups}, where it is called ``L-shellability'' and where chains are read from top to bottom.  We will follow what is now the customary definition of CL-shellability from \cite{bjorner-wachs-lex-shellable-posets}, where chains are instead read from bottom to top.  Because our chain labeling will be read from top to bottom, we will actually show that the dual of the interval $[u,w]$ is CL-shellable and hence shellable; this implies the shellability of $[u,w]$ since the order complex of $[u,w]$ is clearly isomorphic to that of its dual.  In this case, we say that $[u,w]$ is \emph{dual CL-shellable}.  The conditions needed to show that $[u,w]$ is dual CL-shellable are stated at the end of the first paragraph of the proof below of Theorem~\ref{thm-layered-shellable}.  Readers interested in a more detailed exposition of CL-shellability (and a wealth of other information about poset topology) are referred to \cite{wachs-lectures}.

\begin{theorem}\label{thm-layered-shellable}
Let $P$ be a rooted forest.  Suppose an interval $\iuw$ in $P^*$ does not contain a non-trivial disconnected subinterval.  Then $\iuw$ is dual CL-shellable.  
\end{theorem}

Before proving Theorem~\ref{thm-layered-shellable}, it will be helpful to introduce and give relevant terminology for the chain labeling we will use.  We would like to use the position labeling described immediately before Theorem~\ref{thm-layered-disconnected} as our chain labeling.  
Unfortunately, this labeling is too simple to give a CL-labeling, as illustrated by Figure~\ref{fig-bad-cl-interval}(a) for the case $P=\mathbb{P}$, where all three maximal chains are weakly increasing from top to bottom.  
\begin{figure}[htbp]
\begin{center}
\begin{tikzpicture}[scale=1.5]
\tikzstyle{every node}=[shape=circle, inner sep=2pt]; 
\begin{scope}
\draw (1,0) node[draw] (a1) {22};
\draw (0,1) node[draw] (a2) {122};
\draw (1,1) node[draw] (a3) {212};
\draw (2,1) node[draw] (a4) {221};
\draw (1,2) node[draw] (a5) {222};
\draw (a1) -- (a2) --  (a5);
\draw (a1) -- (a3) --  (a5);
\draw (a1) -- (a4) --  (a5);
\draw (0.3,0.5) node {1};
\draw (0.3,1.5) node {1};
\draw (0.9,0.5) node {2};
\draw (0.9,1.5) node {2};
\draw (1.7,0.5) node {3};
\draw (1.7,1.5) node {3};
\draw (1,-0.8) node {(a)};
\end{scope}
\begin{scope}[xshift=4cm]
\draw (1,0) node[draw] (a1) {22};
\draw (0,1) node[draw] (a2) {122};
\draw (1,1) node[draw] (a3) {212};
\draw (2,1) node[draw] (a4) {221};
\draw (1,2) node[draw] (a5) {222};
\draw (a1) -- (a2) --  (a5);
\draw (a1) -- (a3) --  (a5);
\draw (a1) -- (a4) --  (a5);
\draw (0.3,0.5) node {1};
\draw (0.3,1.5) node {1};
\draw (0.86,0.5) node {$\decrease{2}$};
\draw (0.9,1.5) node {2};
\draw (1.7,0.5) node {$\decrease{3}$};
\draw (1.7,1.5) node {3};
\draw (1,-0.8) node {(b)};
\end{scope}
\end{tikzpicture}
\caption{Labeling according to the position decreased as in (a) can result in more than one increasing chain.  In this case, we modify the labels so that only the chain that deletes the leftmost 2 has increasing labels, as in (b).}
\label{fig-bad-cl-interval}
\end{center}
\end{figure}
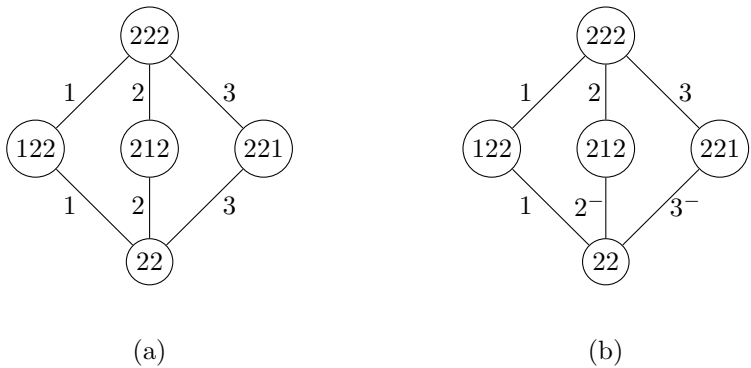
To rectify this situation, we make the following special modification to the position labeling.  Suppose 
$w \covers v \covers u$ and $w$ has a consecutive sequence of $b$'s that is maximal under inclusion, where $b$ is an element of rank 2 in $P_0$.  Since $P$ is a rooted forest, $b$ covers a unique element $a$ in $P$, and $a$ is a minimal element of $P$.  Suppose that the $i$-th of these $b$'s in the consecutive sequence in $w$ is decreased to $a$ in going to $v$ and then that $a$ is deleted in going to $u$.  If $i>1$, then change the label $k$ on $v \covers u$ to $\decrease{k}$, where $k-1 < \decrease{k} < k$ (if we prefer to be specific, $\decrease{k} = k-0.5$ will certainly suffice).  The result is that only the chain that deletes the leftmost $b$ in the consecutive sequence gets weakly increasing labels from top to bottom in $\iuw$.  An example of this modified labeling in the case $P=\mathbb{P}$ is shown in Figure~\ref{fig-bad-cl-interval}(b).  While this modification may seem somewhat arbitrary, we will see in the proof below that it is exactly what we need to get a dual CL-labeling.  We will call the labeling just described the  \emph{modified position labeling}.

\begin{proof}[Proof of Theorem~\ref{thm-layered-shellable}]
We wish to show that the modified position labeling is a dual CL-labeling.  Let $[v,v']_r$ be a top-rooted interval in $\iuw$.  Following the chain $r$ from $w$ to $v'$ gives $v'$ a particular embedding $\emb$ in $w$.  We wish to show that there is a unique increasing maximal chain from $\emb$ to an embedding of $v$ in $\emb$, and that this increasing chain has the lexicographically first labels of all maximal chains in $[v,v']_r$.  None of these conditions to be checked will be affected if we discard any letters of $\emb$ that are zero, and assume that $\emb$ has only nonzero letters.  Therefore, we lose no generality by taking $\emb=w$ and $v=u$.   Thus we will show dual CL-shellability by showing that there is a unique increasing maximal chain from $w$ to an embedding of $u$ in $w$, and that this increasing chain has the lexicographically first labels of all maximal chains in $[u,w]$.

Let a maximal chain $C$ be defined in the following way:  
starting with $w$, decrease the leftmost letter possible such that the result will still be above $u$.  
For example, if $w=2211$ and $u=2$ with $P=\mathbb{P}$, then $C$ is given by
\[
2211 \coverslabel{1} 1211 \coverslabel{1} 0211 \coverslabel{3} 0201 \coverslabel{4} 0200.
\]
We must check several aspects of $C$.
\begin{itemize}
\item Since we are decreasing the leftmost possible letter at each stage, any deletion of a letter from a maximal consecutive sequence of $a$'s, where $a$ is a minimal element of $P$, will respect the convention of deleting the leftmost such $a$.
\item For the same reason, $C$ will eventually arrive at the rightmost embedding $\rho$ of $u$.  
Indeed, suppose $C$ eventually arrived at an embedding $\emb$ of $u$ that was not rightmost, and let $i$ be the leftmost position where $\emb$ differs from the rightmost embedding $\rho$.  Since each letter of an embedding of $u$ is either zero or a particular letter of $u$, and $\rho$ is rightmost, it must be the case that $\rho(i)=0$ and $\emb(i)\neq0$.  This is a contradiction since the definition of $C$ implies that $\emb(i)$ should have been decreased to 0 in this case.
\item For an element $b$ of rank 2 in $P_0$, if we encounter a maximal sequence of consecutive $b$'s and one such $b$ is to be decreased to 0 in two steps, we will always decrease the leftmost such $b$.  In particular, the labels along $C$ will not undergo any of the modifications that change a label $k$ to $\decrease{k}$.
\item Since we always decrease letters as far left as possible, the labels along $C$ will be increasing.  For the same reason, $C$ is clearly the lexicographically least maximal chain in $\iuw$.
\end{itemize}

It remains to show that $C$ is the only increasing chain from $w$ down to $u$.  Consider another chain $C'$ whose labels are increasing.  If $C'$ ends at the rightmost embedding $\rho$ of $u$ in $w$, then $C'$ must decrease the same letters of $w$ as $C$ and by the same amounts.  Since both chains are increasing, $C'$ must then be identical to $C$.  Therefore, suppose $C$ ends at an embedding $\emb$ of $u$ with $\emb \neq \rho$.  Find the rightmost position $j$ at which $\rho$ and $\emb$ differ.  Since each position of an embedding of $u$ is either 0 or a particular letter of $u$, and since $\rho$ is rightmost, it must be the case that $\emb(j)=0$ and $\rho(j) = u(k)  \neq 0$ for some $k$.  Working left from position $j$, the next nonzero entry of $\emb$ must be $\emb(i)=u(k)$ for some $i$.  Note that $w(i), w(j) \geq u(k)$.  The setup for $w$, $\rho$ and $\emb$ can be summarized as 
\[
\arraycolsep=0pt
\begin{array}{rlclcccccccl}
w =  &\ (w(1), &\ldots&, w(i), &\ldots&,w(j-1)& ,& w(j)&,& \ldots& ,w(\parts{w})), \\
\rho=  &\ (\rho(1), &\ldots&, \rho(i), &\ldots&, \rho(j-1)&,& u(k)&,& \rho(j+1), \ldots&, \rho(\parts{w})), \\
\emb=  &\ (\emb(1), &\ldots, \emb(i-1)&, u(k), &0,0,& \ldots &,&0&,&  \rho(j+1), \ldots &,\rho(\parts{w})).
\end{array}
\]
Since $C'$ has increasing labels, during the process of decreasing the letter in position $j$ of $w$, it must at some point encounter elements $v_1$, $v_2$ with $v_1 > v_2$ that embed in $w$ as
\[
\left(\emb(1), \ldots, \emb(i-1), u(k), 0, 0, \ldots, 0, u(k), w(j+1), \ldots, w(\parts{w}) \right)
\]
and
\begin{equation}\label{equ-v2-embedding}
 \left(\emb(1), \ldots, \emb(i-1), u(k), 0, 0, \ldots, 0, 0, w(j+1), \ldots, w(\parts{w}) \right)
\end{equation}
respectively.
There are three cases to consider.  If $u(k)=a$ where $a$ is a minimal element of $P$, then the convention dictates that the $a$ in position $i$ of $v_1$ (or an $a$ even further left) should have been decreased instead of the $a$ in position $j$, contradicting the fact that~\eqref{equ-v2-embedding} is the embedding of $v_2$ corresponding to $C'$.  If $u(k)=b$ where $b$ has rank 2 in $P_0$, then the labels on the edges from $v_1$ down to $v_2$ will be $j, \decrease{j}$ in that order, contradicting the fact that $C'$ is increasing.  If $u(k)>3$, then the open interval $(v_1,v_2)$ is in the situation of (\ref{item-layered-disconnected}) of Theorem~\ref{thm-layered-disconnected}, contradicting our hypothesis that $\iuw$ does not contain a non-trivial disconnected subinterval.  
\end{proof}

\begin{remark}\label{rem-dmt-proof}
When $P$ is a rooted forest, ideas from discrete Morse theory give an alternative way to show that an interval $[u,w]$ in $P^*$ is shellable if it does not contain a non-trivial disconnected subinterval.  
Let us give a sketch of the proof for readers familiar with discrete Morse theory.  Suppose $[u,w]$ contains no non-trivial disconnected subintervals and that, under the position labeling, a maximal chain $C$ from $w$ to $u$ contains an MSI $C(v',v)$ with more than one element.  In particular, $\rk{v'}-\rk{v}\geq3$.  Restrict to the interval $(v, v')$ and discard positions where $v'$ is zero, adjusting the edge labels accordingly.  Then we are in the situation of Condition (4) of Theorem~\ref{thm-layered-disconnected}.  Thus $(v, v')$ is disconnected, contradicting the fact that $[u,w]$ contains no non-trivial disconnected subintervals.  We conclude that all MSIs in $[u,w]$ have just a single element, in which case \cite[Prop.~4.2]{babson-hersh} implies $[u,w]$ is shellable.

Although this discrete Morse theoretic proof certainly has the advantage of being short, it does not give an explicit CL-labeling like our original proof.  A further advantage of our original proof is that it uses more classical ideas, and so might be more accessible to many readers.  One might also speculate that our original proof would have a better chance of being generalized; see Subsection~\ref{subsec-separable} for a discussion of the case of \emph{separable} permutations.
\end{remark}

\begin{remark}
It is worth comparing Theorem~\ref{thm-layered-shellable} to a similar shellability result in the literature that applies to all posets.  As noted by Wachs \cite{wachs}, a result of Billera and Myers \cite{billera-myers} implies that any poset is shellable if it is $(2+2)$-free, meaning it does not contain the 4-element poset consisting of two disjoint 2-element chains as an induced subposet.  The converse result does not hold as shown, for example, by the Boolean lattice of rank 3.  More importantly for us, Theorem~\ref{thm-layered-shellable} is not implied by \cite{billera-myers} since there are examples of intervals in $P^*$ that do not contain non-trivial disconnected subintervals and are not $(2+2)$-free.  In other words, the hypotheses of Theorem~\ref{thm-layered-shellable} apply, but \cite{billera-myers} does not.  One example is $[11,221]$ when $P=\mathbb{P}$, corresponding to the interval $[12,21435]$ in $\clp$.
\end{remark}

As a consequence of shellability, for $P$ a rooted forest, we get that any interval $[u,w] \in P^*$ that does not contain a non-trivial disconnected subinterval is homotopic to a wedge of $|\mu(u,w)|$ spheres, each of the top dimension $\rk{w}-\rk{u}-2$.  Therefore, we know the homotopy type completely since a formula for $\mu(u,w)$ is given in \cite{sagan-vatter}.  A formula for $\mu(u,w)$ for general $P$ is the main result of \cite{mcnamara-sagan}.  Modifying this latter formula for the case of decomposable permutations is the subject of the next section.  
\section{The M\"obius function of decomposable intervals}\label{mob-decomp}

Suppose $\tau$ is a decomposable permutation and let $\tau = \tau_1 \op \cdots \op \tau_t$ be its finest decomposition throughout this section.  Results in \cite[Prop.~1 and~2]{mob-sep} give recurrences  that reduce the computation of the M\"obius function $\mst$ to M\"obius function calculations of the form $\mu(\sig', \tau')$ where $\tau'$ is a single component of $\tau$ and $\sig'$ is a direct sum of consecutive components of $\sig$.  For example, a corollary of these results of \cite{mob-sep} is that if $\sig$ is indecomposable, then $\mst$ is either 0 or $\pm\mu(\sig,\tau_1)$, depending on the form of $\tau$.  

A disadvantage of the results of \cite{mob-sep} is that the recurrences are given in the form of two different propositions, one for the case $\tau_1=\one$ and one for $\tau_1 >\one$; the formulas for $\mst$ in the two propositions look very different, as shown below in Propositions~\ref{prop-bjjs1} and~\ref{prop-bjjs2}.  We now state our new formula, which replaces the two propositions by a single recursive expression for $\mst$.   

\begin{proposition}\label{prop-mobius}
Consider permutations $\sig$ and $\tau$ and let $\tau=\tau_1 \op \cdots \op \tau_t$ be the finest decomposition of $\tau$.  Then
\begin{equation}\label{equ-mobius}
\mu(\sig, \tau) = \sum_{\sig = \sg_1 \op \cdots \op \sg_t}\ \prod_{1 \leq m \leq t} 
\left\{ 
\begin{array}{ll} 
\mu(\sg_m\,,\tau_m) +1 & \mbox{if $\sg_m = \emptyset$ and $\tau_{m-1}=\tau_m$}\,, \\ 
\mu(\sg_m\,,\tau_m) & \mbox{otherwise},
\end{array} \right. 
\end{equation}
where the sum is over all direct sums $\sig = \sg_1 \op \cdots \op \sg_t$ such that $\emptyset \leq \sg_m \leq \tau_m$ for all $1 \leq m \leq t$. 
\end{proposition}

The condition $\tau_{m-1}=\tau_m$ is considered false when $m=1$ since $\tau_0$ does not exist.  Proposition~\ref{prop-mobius} is inspired by, and is an exact analogue of, the formula from \cite{mcnamara-sagan} for the M\"obius function for generalized subword order.  Unfortunately, we have not been able to find a way to obtain Proposition~\ref{prop-mobius} as an application of the formula for generalized subword order.  Instead, we will prove Proposition~\ref{prop-mobius} by showing that it gives the same recursive expressions for $\mu(\sig, \tau)$ as the propositions of \cite{mob-sep}.  Before doing so, let us give an example of Proposition~\ref{prop-mobius}.

\begin{example} As an example of how Proposition~\ref{prop-mobius} can be used, we compute $\mu(12,24136857)=\mu(12\,,\,2413\op2413).$  It is straightforward to compute by hand that $\mu(12, 2413)=3$ and $\mu(\one, 2413)=-3$.  We also know that $\mu(\emptyset, \tau)=0$ for any $\tau>\one$.  On the other hand, $[12\,,\,2413\op2413]$ has 62 elements and 223 edges, meaning that computing $\mu(12, 24136857)$ directly is a much less pleasant exercise.  Instead, applying Proposition~\ref{prop-mobius},     there are three terms in the sum:
\begin{itemize}
\item $12 = \one \op \one$ contributes $\mu(\one,2413)\mu(\one,2413) = 9$;
\item $12 = \emptyset \op 12$ contributes $\mu(\emptyset,2413)\mu(12,2413) = 0$;
\item $12 = 12 \op \emptyset$ contributes $\mu(12, 2413)(\mu(\emptyset, 2413)+1)=3$, with the $+1$ arising because we have $\sg_2=\emptyset$ and $\tau_1 = \tau_2$. 
\end{itemize} 
Therefore $\mu(12, 24136857)=12$.  
\end{example}

For the purposes of comparison and since they are needed in our proof of Proposition~\ref{prop-mobius}, we next give the two propositions from \cite{mob-sep}.  For a finest decomposition $\tau=\tau_1 \op \cdots \op \tau_t$, we will use the notation $\tau_{\leq i} = \tau_1 \op \cdots \op \tau_i$ and $\tau_{>i} = \tau_{i+1} \op \cdots \op \tau_t$, with $\tau_{\geq i}$ defined similarly.  The first proposition covers the case $\tau_1 = \one$.  

\begin{proposition}[Proposition 1 of \cite{mob-sep}]\label{prop-bjjs1}
Let $\sig$ and $\tau$ be nonempty permutations with finest decompositions $\sig = \sig_1 \op \cdots \op \sig_s$ and $\tau=\tau_1 \op \cdots \op \tau_t$, where $t \geq2$.  Suppose that $\tau_1=\one$.  Let $k \geq 1$ be the largest integer such that all the components $\tau_1, \ldots, \tau_k$ are equal to $\one$, and let $\ell \geq 0$ be the largest integer such that all the components $\sig_1, \ldots, \sig_\ell$ are equal to $\one$.  Then
\[
\mst = \left\{ \begin{array}{ll} 
0 & \mbox{if $k-1>\ell$}\,,\\
-\mu(\sig_{>k-1}\,, \tau_{>k}) & \mbox{if $k-1=\ell$}\,,\\
\mu(\sig_{>k}\,,\tau_{>k}) - \mu(\sig_{>k-1}\,,\tau_{>k}) & \mbox{if $k-1<\ell$}\,.
\end{array}
\right.
\]
\end{proposition}

The remaining case is $\tau_1 >\one$ and is covered by the next proposition.

\begin{proposition}[Proposition 2 of \cite{mob-sep}]\label{prop-bjjs2}
Let $\sig$ and $\tau$ be nonempty permutations with finest decompositions $\sig = \sig_1 \op \cdots \op \sig_s$ and $\tau=\tau_1 \op \cdots \op \tau_t$, where $t \geq2$.  Suppose that $\tau_1>\one$.  Let $k \geq 1$ the the largest integer such that all the components $\tau_1, \ldots, \tau_k$ are equal to $\tau_1$.  Then
\[
\mst = \sum_{i=1}^s \sum_{j=1}^k \mu(\sig_{\leq i}\,,\tau_1)\mu(\sig_{>i}\,,\tau_{>j}).
\]
\end{proposition}

Since reversal of permutations preserves containment, all three propositions remain true when decompositions are replaced by skew decompositions and direct sums by skew sums.  

Although Propositions~\ref{prop-bjjs1} and~\ref{prop-bjjs2} as stated in \cite{mob-sep} require that $\tau$ be decomposable, we can check that they also give correct expressions for the M\"obius function even when $\tau$ is indecomposable, i.e., $t=1$.  This allows us to use $t=1$ as the base case in the induction parts of the proof below.  Observe also that the decomposition $\sig=\sg_1 \op \cdots \op \sg_t$ appearing in Proposition~\ref{prop-mobius} has the same number of components as the finest decomposition of $\tau$ but is otherwise arbitrary and can include empty components.  On the other hand, the decomposition of $\sig$ appearing in Propositions~\ref{prop-bjjs1} and~\ref{prop-bjjs2} is the finest decomposition.  This difference is the reason for our choice of different characters for the components of the two decompositions.

\begin{proof}[Proof of Proposition~\ref{prop-mobius}]
We first consider the case $\tau_1=\one$ and adopt the notation of Proposition~\ref{prop-bjjs1}.  Suppose first that $k-1 >\ell$.  Then in every decomposition $\sg=\sg_1 \op \cdots \op \sg_t$ of Proposition~\ref{prop-mobius}, there must exist $m$ with $2 \leq m \leq k$ such that $\sg_m=\emptyset$ and $\tau_{m-1}=\tau_m=\one$.  This $m$ will contribute $-1+1$ to the product of~\eqref{equ-mobius}, consistent with Proposition~\ref{prop-bjjs1}.  

If $k-1=\ell$, for $\sig=\sg_1 \op \cdots \op \sg_t$ to contribute a nonzero amount to the sum, it must be the case that $\sg_1=\emptyset$ and $\sg_2 = \cdots = \sg_k=\one$ to avoid the situation of the previous paragraph.  We first note that if $k=t$, then Propositions~\ref{prop-mobius} and~\ref{prop-bjjs1} give equal values for $\mst$.  From here on, it will be helpful to abbreviate the expression 
\begin{equation}\label{equ-twocases}
\left\{ 
\begin{array}{ll} 
\mu(\sg_m\,,\tau_m) +1 & \mbox{if $\sg_m = \emptyset$ and $\tau_{m-1}=\tau_m$}\,, \\ 
\mu(\sg_m\,,\tau_m) & \mbox{otherwise}
\end{array} \right.
\end{equation}
from Proposition~\ref{prop-mobius} by $\twocases$.
For $k<t$, Proposition~\ref{prop-mobius} gives
\begin{align*}
\mu(\sig, \tau) =  & \sum_{\sig = \sg_1 \op \cdots \op \sg_t}\ \prod_{1 \leq m \leq t} 
\twocases \\
=& \sum_{\sig = \emptyset \op \one \op \cdots \op \one \op \sg_{k+1} \cdots \op \sg_t} (-1)(+1)^{k-1} \prod_{k+1 \leq m \leq t} 
\twocases \\
=& -\sum_{\sg_{k+1}\op \cdots \op \sg_t}\ \prod_{k+1 \leq m \leq t} 
\twocases \\
=& -\mu(\sig_{>k-1}\,, \tau_{>k}),
\end{align*}
with the last equality being by induction in Proposition~\ref{prop-mobius} on $t$, the number of components in the finest decomposition of $\tau$.  Proposition~\ref{prop-mobius} clearly holds when $t=1$.

Now suppose $k-1 < \ell$.  For $\sig=\sg_1 \op \cdots \op \sg_t$ to contribute a nonzero amount to~\eqref{equ-mobius}, we again require that $\sg_2 = \cdots = \sg_k=\one$ but we can now have $\sg_1 = \emptyset$ or $\sg_1 = \one$.  The first possibility will contribute $-\mu(\sig_{>k-1}\,, \tau_{>k})$ as above.  A very similar calculation shows that the second possibility will contribute $\mu(\sig_{>k}, \tau_{>k})$.

We now consider the trickier case $\tau_1>\one$ and refer to Proposition~\ref{prop-bjjs2}.  In an embedding of $\sig$ in $\tau$, we will have the portion $\sig_1 \op \cdots \op \sig_i$ of $\sig$ embedding in $\tau_1$ for some $0 \leq i \leq s$.  Moving to the setting of Proposition~\ref{prop-mobius}, this situation corresponds to $\sg_1 = \sig_1 \op \cdots \op \sig_i$.  If $i=0$, then $\sg_1=\emptyset$ in Proposition~\ref{prop-mobius} and it will contribute $\mu(\emptyset, \tau_1)=0$ to the product in~\eqref{equ-mobius}.  So we can assume $1 \leq i \leq s$ and the right-hand side of~\eqref{equ-mobius} becomes
\begin{equation}\label{equ-mobius2}
\sum_{i=1}^s \mu(\sig_{\leq i}\,,\tau_1) \sum_{\sig_{>i} = \sg_2 \op \dots \op \sg_t}\ \prod_{2 \leq m\leq t} 
\twocases.
\end{equation}
Next, consider the fact that we must have $\sg_2 = \sg_3 = \cdots =\sg_j = \emptyset$ for some maximal $j$ with $1 \leq j \leq t$ (where $j=1$ just means that $\sg_2 \neq \emptyset$).  There are two cases to consider, namely $j\leq k$ and $j > k$.

If $j \leq k$ then the contribution of $m$ with $2 \leq m \leq j$ to the product in~\eqref{equ-mobius2} will be $\mu(\emptyset,\tau_1)+1 = 1$, and so these values can be ignored in the product.  Therefore, the portion of~\eqref{equ-mobius2} corresponding to the $j\leq k$ case can be written as
\begin{equation}\label{equ-small-j-contrib}
\sum_{i=1}^s \mu(\sig_{\leq i}\,,\tau_1) \sum_{j=1}^k\ \sum_{\sig_{>i} = \sg_{j+1} \op \dots \op \sg_t}\ \prod_{j+1 \leq m\leq t} 
\twocases,
\end{equation}
with the additional condition on the third sum that $\sg_{j+1} \neq \emptyset$.  

If $j>k$, then the contribution of $m$ with $2 \leq m \leq k$ to the product in~\eqref{equ-mobius2} can be ignored like before.  Therefore, the portion of~\eqref{equ-mobius2} corresponding to the $j>k$ case can be written as
\begin{equation}\label{equ-big-j-contrib}
\sum_{i=1}^s \mu(\sig_{\leq i}\,,\tau_1) \sum_{\sig_{>i} = \sg_{k+1} \op \dots \op \sg_t}\ \prod_{k+1 \leq m\leq t} 
\twocases,
\end{equation}
now with the additional condition on the second sum that $\sg_{k+1}=\emptyset$.  

Combining~\eqref{equ-small-j-contrib} and~\eqref{equ-big-j-contrib}, we can rewrite~\eqref{equ-mobius2} as
\begin{equation}\label{equ-j-combined}
\sum_{i=1}^s \mu(\sig_{\leq i}\,,\tau_1) \sum_{j=1}^k\ \sum_{\sig_{>i} = \sg_{j+1} \op \dots \op \sg_t}\ \prod_{j+1 \leq m\leq t} 
\twocases,
\end{equation}
with the additional condition on the third sum that $\sg_{j+1} \neq \emptyset$ when $j<k$.  This additional condition ensures that the condition ``$\sg_m = \emptyset$ and $\tau_{m-1}=\tau_m$'' of~\eqref{equ-twocases} will never be satisfied by the first term of the product in~\eqref{equ-j-combined}, just like in the expression for $\mst$ of Proposition~\ref{prop-mobius}.  By induction on the number of components~$t$ in the finest decomposition of $\tau$, \eqref{equ-j-combined} becomes
\begin{equation}\label{equ-bjjs2-as-required}
\sum_{i=1}^s \mu(\sig_{\leq i}\,,\tau_1) \sum_{j=1}^k \mu(\sig_{>i}\,, \tau_{>j}),
\end{equation}
as required. 

It is easily checked that Propositions~\ref{prop-mobius} and~\ref{prop-bjjs2} both give $\mu(\sig, \tau_1)$ in the base case $t=1$ of the induction.  An incisive reader may notice that the argument above has the potential to run into technical difficulties in~\eqref{equ-small-j-contrib}, \eqref{equ-big-j-contrib} and \eqref{equ-j-combined} in the case when $k=t$, i.e., $\tau = \tau_1 \op \cdots \op \tau_1$.  The proof above will work fine in this case except when $j\geq k$, which amounts to $j=k$ since $j \leq t$.  Note that $j=t$ then also dictates that $i=s$ by the original definition of $j$.  In this situation, following through our ideas from above, the portion of~\eqref{equ-mobius2} corresponding to $i=s$ and $j=t=k$ is $\mu(\sig, \tau_1)$, which remains consistent with~\eqref{equ-bjjs2-as-required}.
\end{proof}

\section{Open problems}\label{sec-open}

\subsection{Preservation of disconnectivity under diminution}

It is natural to wonder if any converse results exist for Corollary~\ref{cor-augmentation}.  For example, suppose $\tau$ can be decomposed as $\tau = \tau_1 \op\tau_2 \op\cdots\op\tau_t$ and $\sig=\tau_1 \op\sig'$ for some $\sig'$.  Is it true that if $(\sig,\tau)$ is disconnected, then so is 
$(\sig'\,,\, \tau_2\op\cdots\op\tau_t)$?  The answer is ``no'' due, for example, to the fact that $(321\,,\, 321\op321)$ is disconnected, but $(\emptyset, 321)$ is not.  The answer is still ``no'' if we insist that $\sig'\neq\emptyset$, since $(231\op312\,,\,231\op231\op312)$ is disconnected, but $(312\,,\,231\op312)$ is not.  However, in the previous example, if instead of deleting the 231 from the front, we delete the 312 from the end to yield $(231\,,\, 231\op231)$, then disconnectivity is preserved.  The answer to the following question is ``yes'' for all $\len{\tau} \leq 10$, and for layered permutations by Theorem~\ref{thm-layered-disconnected}. 

\begin{question}\label{que-decomposable}
Suppose $\sig$ and $\tau$ are decomposable permutations with $\len{\tau}-\len{\sig}\geq3$ and 
with finest decompositions $\sig=\sig_1 \op \cdots\op\sig_s$ and $\tau=\tau_1\op\cdots\op\tau_t$.  If $(\sig,\tau)$ is disconnected, then is at least one of the following two statements true:
\begin{itemize}
\item $\sig_1 = \tau_1$ and $(\sig_2\op\cdots\op\sig_s\,,\,\tau_2\op\cdots\op\tau_t)$ is disconnected;
\item $\sig_s = \tau_t$ and $(\sig_1\op\cdots\op\sig_{s-1}\,,\,\tau_1\op\cdots\op\tau_{t-1})$ is disconnected?
\end{itemize} 
\end{question}

Note that this question does not just ask about preservation of disconnectivity under deletion of certain elements, but also asks if the finest decompositions have matching first or last parts when $(\sig,\tau)$ is disconnected.  The answer is ``no'' if we allow $\len{\tau}-\len{\sig}=2$, as shown by the interval $(12, 2143)$.  An affirmative answer to Question~\ref{que-decomposable} would imply that all disconnected $(\sig,\tau)$ with $\sig$ and $\tau$ decomposable and $\len{\tau}-\len{\sig}\geq3$ can be viewed as consequences of parts (a) and (c) of Corollary~\ref{cor-augmentation}.  Obviously, a similar question can be asked about skew decompositions.

\subsection{Non-shellable intervals without disconnected subintervals}

In view of Theorem~\ref{thm-layered-shellable}, it is natural to ask if there exist intervals $[\sig,\tau]$ that are not shellable but have no non-trivial disconnected subintervals.  While we do not have a good way to  test shellability computationally, 
we can test whether a poset is Cohen-Macaulay, i.e., whether all the homology is in the top dimension, which is implied by shellability.  The first intervals $\ist$ that have no non-trivial disconnected subintervals but are not Cohen-Macaulay, and thus not shellable, occur when $\len{\tau}=7$.  One such example is $[123, 3416725]$.  It would be interesting to determine if there is something simple about the structure of such intervals that implies their non-shellability.  

\subsection{Separable permutations}\label{subsec-separable}

By Theorem~\ref{thm-layered-shellable}, we know that an interval of layered permutations of rank at least 3 is shellable if and only if it does not contain any non-trivial disconnected subintervals.  Does the same property hold for any larger class of intervals?  It does not hold in general for $[\sig,\tau]$, and not even with $\sig$ and $\tau$ decomposable, since $[\one \op 123, \one\op 3416725]$ is not shellable but has no non-trivial disconnected subintervals.  Moreover, the interval $[\one\op123,\one\op3416725]$ is \emph{not} isomorphic to $[123,3416725]$, so the non-shellability of the former interval is not a trivial consequence of the non-shellability of the latter one (where 3416725 is indecomposable).

Layered permutations are special cases of \emph{separable} permutations.  A permutation is separable if it can be generated from the permutation $\one$ by successive sums and skew sums.  In other words, a permutation is separable if it is equal to $\one$ or can be expressed as the sum or skew sum of separable permutations.  For example, $52143 = \one \om ((\one\om\one)\op(\one\om\one))$.  Equivalently, a permutation is separable if it avoids the patterns 2413 and 3142 (see \cite{bose-buss-lubiw}).  Consequently, if $\tau$ is separable, then any $\sig \leq \tau$ is also separable.  

\begin{conjecture}\label{conj-sep-shell}
An interval $[\sig,\tau]$ of separable permutations with $\len{\tau}-\len{\sig} \geq 3$ is shellable if and only if it has no non-trivial disconnected subintervals.  
\end{conjecture}

It was shown in \cite[Cor. 24 and 25]{mob-sep} that for a separable permutation $\tau$, the M\"obius function $\mu(\one,\tau)$ can only take the values 0, 1 and $-1$, and that the same is true of $\mst$ if $\sig$ occurs precisely once in $\tau$.  If true, Conjecture \ref{conj-sep-shell} would therefore imply that, for such $\sig$ and $\tau$, intervals $[\one,\tau]$ and $\ist$ are each either contractible or homotopy equivalent to a single sphere (of dimension $\len{\tau}-3$ and $\len{\tau}-\len{\sig}-2$, respectively).
  
As in Theorem~\ref{thm-disconnected-subinterval},
the ``only if'' direction of Conjecture~\ref{conj-sep-shell} is known.  The ``if'' direction holds for $[\sig,\tau]$ of rank 3, since shellability of such $[\sig,\tau]$ is equivalent to connectivity of $(\sig,\tau)$.   As other evidence in favor of the ``if'' direction, we have found, by computer tests, that all such intervals with $\len{\tau}\leq9$ are Cohen-Macaulay.  A weaker condition than $[\sig,\tau]$ being Cohen-Macaulay is that the M\"obius function alternates in sign, i.e., the sign of every subinterval of $[\sig,\tau]$ is $(-1)^r$ where $r$ is the rank of the subinterval \cite[Prop.~3.8.11]{ec1,ec1e2}.  We have checked that if $[\sig,\tau]$ has no {non-trivial disconnected subintervals}, then the M\"obius function of $[\sig,\tau]$ alternates in sign whenever $\len{\tau}\leq 10$ and also for $\len{\sig}=7$ when $\len{\tau}=11$.

An obvious question is whether the proof of Theorem~\ref{thm-layered-shellable} could be extended to separable permutations.  As it happens, the proof of $(2) \Rightarrow (1)$ of Theorem~\ref{thm-layered-disconnected} goes through in the case of separable permutations, which follows from Lemma~\ref{lem-sum-disconn} and Corollary~\ref{cor-augmentation}, but we have been unable to characterize shellability in the case of intervals of separable permutations.

\begin{lemma}\label{lem-separable-disconnected}
Let $\sig$ and $\tau$ be separable permutations with $\len{\tau}-\len{\sig}\geq3$.  Suppose $\tau$ has a contiguous subword of contiguous letters that, after flattening, takes the form $\pi\op\pi$ with $\pi$ indecomposable or $\pi\om\pi$ with $\pi$ skew indecomposable.  Suppose $\sig$ is obtained from $\tau$ by removing one of these copies of {$\pi$}.  Then $(\sig,\tau)$ is disconnected.   
\end{lemma}

We can also ask if the converse of Lemma~\ref{lem-separable-disconnected} is  true, although in the layered case, the corresponding statement was not needed in the proof of Theorem~\ref{thm-layered-shellable}.  

One difficulty of extending the proof of Theorem~\ref{thm-layered-shellable} seems to be that the idea of the rightmost embedding does not extend immediately to separable permutations.  For example, is 10002 or 01200 the ``correct'' rightmost embedding of 12 in 14532?

\subsection{Structural questions and the consecutive pattern poset}

Let us say that $\sig$ occurs in $\tau$ as a \emph{consecutive pattern} if there is a subsequence of \emph{consecutive} letters of $\tau$ that appear in the same relative order of size as those in $\sig$.  For example, 352 is an occurrence of the consecutive pattern 231 in 43521, whereas 452 is not.  The consecutive pattern poset $\cpp$ is then the obvious analogue of $\clp$ for consecutive patterns.  One indication that $\cpp$ is more tractable than $\clp$ is that every element of $\cpp$ covers at most two elements.  Another indication is that stronger results have been obtained on the structure of $\cpp$ than of $\clp$.  In particular, the M\"obius function of all intervals of $\cpp$ has been determined in \cite{bernini-ferrari-steingrimsson, sagan-willenbring}. 

The goal of \cite{elizalde-mcnamara} has been to see to what extent the results of the present work could be extended to $\cpp$ and, as one would hope and might even expect, stronger results can be obtained in the setting of $\cpp$.  The first is that the analogue of Theorem~\ref{thm-layered-shellable} is true for \emph{all} intervals in $\cpp$ not containing a disconnected subinterval, not just those formed by layered permutations.  The statement of Theorem~\ref{thm-disconnected-subinterval} carries through to $\cpp$ verbatim.  In addition, there are two results in \cite{elizalde-mcnamara} whose analogues in $\clp$ are open questions.  Let us call the number of elements of a given rank the \emph{size} of the rank.  The first result is that all intervals of $\cpp$ are rank-unimodal, meaning that the rank sizes read from bottom to top form a sequence of the form $a_0 \leq a_1 \leq \cdots \leq a_k \geq a_{k+1} \geq \cdots \geq a_{\ell}$.  We have checked the following corresponding assertion for all intervals $[\sig,\tau]$ in $\clp$ with $|\tau| \leq 8$.

\begin{conjecture}
Every interval $[\sig,\tau]$ in $\clp$ is rank-unimodal.
\end{conjecture}

To explain the second result from \cite{elizalde-mcnamara} that could possibly carry over to $\clp$, we need some definitions.  Recall that a ranked poset $P$ is said to be \emph{Sperner} if the size of the largest antichain equals the largest rank size.   A \emph{$k$-family} of $P$ is a union of $k$ antichains.  If $P$ has rank $r$ and $1 \leq k \leq r+1$, we say that $P$ is \emph{$k$-Sperner} if the sum of the $k$ largest rank sizes equals the size of the largest $k$-family.  Finally, $P$ is said to be \emph{strongly Sperner} if it is $k$-Sperner for all $k$.  Since all intervals in $\cpp$ are strongly Sperner \cite{elizalde-mcnamara}, is is natural to ask the following questions.

\begin{questions}
Are all intervals in $\clp$ Sperner?  If so, are they strongly Sperner?
\end{questions}

\bibliographystyle{abbrv}
\bibliography{topology-perm-poset}

\end{document}